\documentclass[jsl]{amsart}

\usepackage{a4wide}

\usepackage{amssymb}
\usepackage{amsfonts}
\usepackage{amsmath}
\usepackage{latexsym}

\usepackage[all]{xy}
\usepackage{graphics}
\usepackage{graphicx}

\usepackage[english]{babel}
\usepackage[utf8]{inputenc}
\usepackage{algorithm}
\usepackage[noend]{algpseudocode}

\algnewcommand\algorithmicswitch{\textbf{either}}
\algnewcommand\algorithmiccase{\textbf{option}}
\algnewcommand\algorithmicnocase{\textbf{}}
\algdef{SE}[SWITCH]{Switch}{EndSwitch}[1]{\algorithmicswitch\ #1}{\algorithmicend\ \algorithmicswitch}%
\algdef{SE}[CASE]{Case}{EndCase}[1]{\algorithmiccase\ #1}{\ \algorithmiccase}%
\algdef{SE}[NOCASE]{Nocase}{EndNocase}[1]{\algorithmicnocase\ #1}{\algorithmicend\ \algorithmicnocase}%
\algtext*{EndSwitch}%
\algtext*{EndCase}%

\newlength\myindent
\newenvironment{proofsk}{\noindent{\bf Sketch of Proof}\hspace*{1em}}{\qed\bigskip\\ }
\usepackage{tikz}

\usetikzlibrary{decorations.markings}
\tikzset{->-/.style={decoration={
  markings,
  mark=at position .6 with {\arrow[scale=2]{>}}},postaction={decorate}}}
\tikzset{-<>-/.style={decoration={
  markings,
mark=at position .3 with {\arrow[scale=2]{<}},
  mark=at position .7 with {\arrow[scale=2]{>}}
},postaction={decorate}}}

\usetikzlibrary{decorations.pathreplacing}

\def\dom{{\sf dom}}

\def\Ax{{\bf Ax}}

\def\set#1{\{#1\}}

\def\F{\f F}

\DeclareMathOperator{\G}{\Gamma_{\mathit{D}}}

\newtheorem{Thm}{Theorem}[section]
\newtheorem{Def}[Thm]{Definition}
\newtheorem{Pro}[Thm]{Proposition}
\newtheorem{Lem}[Thm]{Lemma}
\newtheorem{Cor}[Thm]{Corollary}
\newtheorem{Rem}[Thm]{Remark}
\newtheorem{Cla}[Thm]{Claim}

\newtheorem{Prob}[Thm]{Problem}
\newtheorem{Conj}[Thm]{Conjecture}

\newenvironment{definition}
{\begin{Def}\sl}{\end{Def}}
\newenvironment{theorem}
{\begin{Thm}\sl}{\end{Thm}}
\newenvironment{claim}
{\begin{Cla}\sl}{\end{Cla}}

\newenvironment{lemma}
{\begin{Lem}\sl}{\end{Lem}}
\newenvironment{corollary}
{\begin{Cor}\sl}{\end{Cor}}

\def\props{{\sf props}}

\def\F{{\bf F}}
\def\G{{\bf G}}
\def\P{{\bf P}}
\def\H{{\bf H}}

\newcommand{\nats}{\mbox{\( \mathbb N \)}}

\newcommand{\rats}{\mbox{\(\mathbb Q\)}}
\newcommand{\reals}{\mbox{\(\mathbb R\)}}
\newcommand{\ints}{\mbox{\(\mathbb Z\)}}

\def\c#1{{\mathcal #1}}

\def\restr #1{{\restriction_{#1}}}

\def\to{\rightarrow}

\def\x{{\sf x}}
\def\y{{\sf y}}
\def\z{{\sf z}}
\def\w{{\sf w}}

\def\e{{\sf e}}

\def \b {{\sf b}}
\def \t  {{\sf t}}
\def \l {{\sf l}}
\def \r{{\sf r}}

\def\l{{\sf l}}
\def\r{{\sf r}}

\def\Sq{{\sf Sq}}

\def\MCS {{\sf MCS }}
\def\MCSx{{\sf MCS}}

\def\MCSs {{ \sf MCS}s }

\def\MCSss {{ \sf MCS}s}

\def\trace{t}
\def\triang{\tau}

\def\move#1#2#3{(1, #1, #2, #3)}
\def\movel#1#2{(2, #1, #2)}

\author{R Hirsch} 

\author{M Reynolds}

\thanks{Thanks to B. McLean, I. Shapirovski, V. Shehtman and S. Kikot for valuable contributions to this paper.}  


\title[Temporal Logic of Minkowski spacetime is decidable]{The Temporal Logic of two dimensional Minkowski spacetime is decidable} 

\begin{document}

\begin{abstract}
We consider Minkowski spacetime, the set of all point-events of spacetime
under the relation of {\em causal} accessibility. That is, $\x$ can access $\y$ if an
electromagnetic or (slower than light) mechanical signal could be sent from $\x$ to $\y$.
We use Prior's tense language of $\F$ and $\P$ representing causal accessibility
and its converse relation.  We consider two versions, one where the accessibility relation is reflexive and one where it is irreflexive.

In either case it has been an open problem, for decades, whether the logic is decidable or  axiomatisable.
We make a small step forward by
proving, in each case, that the set of valid formulas over two-dimensional Minkowski 
spacetime is decidable and that the complexity of each problem is PSPACE-complete.

A consequence is that the temporal   logic of intervals with real endpoints under either the containment relation or the strict containment relation is PSPACE-complete, the same is true if the interval accessibility relation is ``each endpoint is not earlier'', or its irreflexive restriction.

We provide a temporal formula that distinguishes between three-dimensional and two-dimensional Minkowski spacetime and another temporal formula that distinguishes the two-dimensional case where the underlying field is the real numbers from the case where instead we use the rational numbers.

\end{abstract}

\maketitle
\section{Introduction}
Arthur Prior, pioneer of tense and temporal logic,
realised that
relativity challenged 
some of the basic assumptions that
seem to underlie such logics. He mentions relativistic time briefly in his monograph \cite{Pri67} as an example of a non-linear flow of time and
his last published talk before his untimely
death
in Norway
was on this very subject
\cite{PriorNOP70}.
See \cite{Mul07} for discussion of the
interplay of tense logic
and relativity at that time.

Goldblatt studied the Modal Logic of Minkowski Spacetime \cite{Gol80} and showed that the modal logic for such frames was axiomatised by {\bf S4.2} (with axioms for reflexivity, transitivity and confluence), regardless of the (non-zero) number of spatial dimensions, where the reflexive accessibility relation is either `can send a signal at the speed of light or less' or `can send a signal at less than the speed of light'.   The same set of axioms was shown to be complete  over reflexive frames where the coordinates are taken from the rational numbers, rather than the reals.  However he observed that irreflexive frames of this type could be distinguished by modal formulas, the logic of irreflexive slower than light accessibility was shown to be axiomatised by {\bf OI.2} (with axioms for transitivity, seriality, confluence and two-density) by
Shehtman and Shapirovsky 
\cite{DBLP:conf/aiml/ShapirovskyS02}.

The main focus of this paper is the logic of valid \emph{temporal formulas} over two dimensional Minkowski frames (i.e. one spatial and one temporal dimension).  The frames for these temporal logics are the same as the corresponding modal frames but the temporal language extends the modal language because it includes an operator $\P$ to refer to points in the past, as well as the purely modal $\F$ operator to refer to the future.    One key difference between the two dimensional and higher dimensional cases is that the reflexive `can send a signal at the speed of light or less' relation in two dimensions is a  distributive lattice while it fails to be a lattice in higher dimensions.  (There are other peculiarities of two dimensional Minkowski spacetime, for example the basic axioms of special relativity permit faster than light travel in two dimensions but not in higher dimensions \cite[theorem~11.7]{AMN07}.)  This peculiarity of the two dimensional case allows us to change the axes of a two dimensional frame and view the frame as a kind of product, as we will see below,  and suggests a more general study of this kind of product.    Let $\c F_1=(F_1, R_1),\; \c F_2=(F_2, R_2)$ be two Kripke frames.  Define the  frame $\c F_1 \bullet  \c F_2$ to be the frame with base set $F_1\times F_2$ and  accessibility relation  $R$ defined by
\[ (x_1, x_2) R(y_1, y_2) \;\iff\; (x_1R_1y_1\wedge x_2R_2y_2)\]
Note that this is a Kripke frame with a single accessibility relation, $R$, unlike the \emph{product frame} more widely used by Modal logicians which has two accessibility relations, one for each dimension.  In the current paper we prove the decidability of the temporal validities of $(\reals^2, \leq) = (\reals, \leq)\bullet(\reals, \leq)$, as well as the irreflexive restriction of this accessibility relation over $\reals^2$.     Other products should also be considered, e.g. ${(\rats^2, \leq),} \;   {(\ints^2,<),}$ \/ ${(\reals,<)\bullet(\ints, \leq)}$ and we discuss some of these briefly towards the end.

If we restrict our two dimensional frames to pairs $(x, y)$ where $x<y$ we may view our restricted frames as \emph{interval} frames.   A modal logic of intervals was first proposed by Halpern and Shoham and shown to be undecidable, for various flows of time \cite{HS91}.  The decidability of various fragments of  interval logics has been investigated quite intensively --- see, for example,  \cite{BDGMS08,MM11}.   The relation `can send a message at the speed of light or less', when restricted to the half-plane $x<y$, becomes the disjunction of the relations $\{\mbox{equals, starts, ended\_by, overlaps, meets},$
 before$\}$ on intervals.  A consequence of our results is that the \emph{temporal} logic of intervals over the reals with this single accessibility relation is decidable, indeed PSPACE-complete.  By a similar consideration of the half-plane $x+y>0$ we prove that the temporal logic of intervals under the \emph{containment} relation (whether strict or not) is also PSPACE-complete.   A further relevant but contrasting result is that the modal logic of two dimensional Minkowski frames with the accessibility relation `can send a signal at exactly the speed of light' is undecidable \cite{Shap10}.

There is some discussion of 
spacetime temporal and modal logics based on the natural numbers
in \cite{Phil01,Phil98,Uck07}.  In addition the logics of frames of higher dimension is also of interest and we conjecture that these logics are all undecidable; however this problem remains open.

\medskip

In this paper our main focus is for two dimensional spacetime with real numbers for time and space.  We prove decidability by a method with some similarities to the mosaic method \cite{Nem95} but here defined in two dimensions.   In very brief outline,  in order to determine the satisfiability of a temporal formula $\phi$ we focus our attention on  the set $Cl(\phi)$ of subformulas and negated subformulas of $\phi$, as is normal in filtration constructions, and consider maximal consistent sets and clusters of maximal consistent sets from this closure set.
We consider a two dimensional model $(\reals^2, <, h)$ of $\phi$, where $h$ maps points in $\reals^2$ to maximal consistent sets, satisfying certain conditions, as a special case of a \emph{rectangle model} $(D, <, h)$ where $D$ is some rectangular subset of $\reals^2$.   The idea is to glue rectangle models together to make more complex rectangle models, in various ways.  Two rectangle models may be glued together if their domains contain a common boundary line and they agree along that line, but more complex combinations of rectangles are also required in our construction (we also define \emph{limits} and \emph{shuffles}), see \cite{FMR12,FMR15} for similar constructions in one dimension.  In order to glue these rectangle models together, only a finite amount of information need be known --- we need to know the sequence of clusters and maximal consistent sets holding  along the boundary of $D$ (in the case where $D$ includes boundary points), but we do not need to know where clusters or maximal consistent sets are true over the interior of $D$, except we do record the minimum and maximum clusters that holds over the interior of $D$.  This information is recorded in what we call a \emph{boundary map}, and there are only finitely many possible boundary maps, given a temporal formula $\phi$.   We give a recursive definition of a set $B$ of these  boundary maps --- every boundary map in $B$ is built recursively out of simpler boundary maps in $B$ --- and we show that every rectangle model gives rise to a boundary map in $B$ and conversely every boundary map in $B$ can be obtained from a rectangle model (lemma~\ref{lem:main}).  Hence, the recursive definition of $B$ provides a decision procedure for the satisfiability of $\phi$.

The structure of this paper is as follows.
Section 2 is on preliminaries,
Section 3 is on Boundary Maps, 
Sections 4 and 5 are on Decidability and Complexity and contain the substantial proofs and main results of this paper,
Section 6 is on Temporal Logics of Intervals and Section 7 provides formulas that distinguish two dimensional frames from higher dimensional frames and real valued frames from rational valued frames.  We end with a list of open problems.
\bigskip

\section{Preliminaries}

\subsection{Formulas}
Temporal propositional formulas are defined by
\[\phi::=\props|\neg\phi|(\phi\vee\phi')|\F\phi|\P\phi\]
where $\props$ is a countably infinite set of propositions, we use standard abbreviations $\alpha\wedge \beta=\neg(\neg\alpha\vee\neg \beta),\;\alpha\rightarrow\beta =\neg\alpha\vee\beta,\;\G(\phi)=\neg\F\neg\phi$ and $\H(\phi)=\neg\P\neg\phi$.    We write $|\phi|$ for the length of, or number of characters in, $\phi$.

\subsection{Frames and Models}\label{sec:prelims}
We consider Kripke frames $(F, R)$, where $R$ is a binary relation over the set $F$.  A \emph{valuation} over such a frame is a map $v:\props\rightarrow\wp(F)$.  A  Kripke structure is a triple $(F, R, v)$ where $v$ is a valuation.   Temporal formulas may be evaluated at points in such structures by using the valuation to evaluate propositions, evaluating propositional connectives in the normal way,  letting $(F, R, v), \x \models\F\phi\iff\exists \y\in F\; ((\x, \y)\in R\wedge (F, R, v), \y\models\phi)$ and $(F, R, v), \x\models\P
\phi\iff \exists \y\in F\; ((\y, \x)\in R\wedge (F, R, v), \y\models\phi)$.  A formula $\phi$ is valid over $(F, R)$, and we may write $(F, R)\models\phi$,  if for all valuations $v$ and all $\x\in F$ we have $(F, R, v), \x\models\phi$.  

\begin{definition}\label{def:rel}
Given a  structure $(F, R, v)$ and a subset $P\subseteq F$ we define the relativized substructure $(P, R_P, v_P)$ where $R_P=R\cap(P\times P)$, and $v_P$ is the valuation defined by $v_P(q)=v(q)\cap P$, by restriction to points in $P$.

Let $\phi$ be a temporal formula not involving the proposition $p$.  The \emph{relativisation} $\phi_p$ of $\phi$ to $p$ is defined by letting $q_p=p\wedge q$, for any proposition $q\neq p$, and
\begin{align*}
(\neg \alpha)_p&=p\wedge\neg(\alpha_p)&
(\alpha\vee\beta)_p&=\alpha_p\vee\beta_p\\
(\F\alpha)_p&=p\wedge\F\alpha_p&
(\P\alpha)_p&=p\wedge \P \alpha_p.
\end{align*}
\end{definition}
The following lemma follows directly from the definition of this relativised formula, we omit the proof.
\begin{lemma}\label{lem:P}
Let $(F, R, v)$ be a structure, let $\phi$ be a temporal formula, let $p$ be a proposition not occurring in $\phi$ and let $P=v(p)$.    Then $(F, R, v), \x \models\phi_p\iff (P,  R_P, v_P), \x\models \phi$, for all $\x\in P$.  Furthermore, if $w(q)\cap P=v(q)\cap P$ for each proposition $q$ of $\phi$ then for any $\x\in P$ we have $(P, R_P, v_P), \x\models\phi\iff (P, R_P, w_P), \x\models\phi$.
\end{lemma}

\subsection{Two-Dimensional Minkowski spacetime}
Points in 2 dimensional Minkow-ski spacetime may be given coordinates $(u, t)$ where $u, t\in \reals$ ($u$ is the spatial coordinate measured in light seconds, $t$ is the time coordinate, in seconds).   Under the ordering `can send a message at the speed of light or less' we may order such points by 
\[ (u, t)\leq (u', t')\iff |u'-u|\leq t'-t \]
It is convenient, however, when working in 2 dimensional Minkowski spacetime to use different axes for the coordinates.  If we take two lightlines through
the origin, one heading right, the other left and let $(x, y)$ represent a point a distance $x$ parallel to the rightgoing axis and $y$ parallel to the leftgoing axis, in other words $x=\sqrt{\frac12}(t+u),\;y=\sqrt{\frac12}(t-u)$,  then the ordering becomes the product order
$$ (x, y)\leq (x', y')\iff x\leq x'\wedge y\leq y'.  $$
It follows that  $(\reals^2, \leq)$ is a  distributive lattice, since it is the product of two distributive  lattices.  This fails spectacularly in higher dimensions.   For example in 3 dimensions using coordinates $(u, v, t)$ where $u, v$ are the spatial coordinates, the future lightcone of the point $(-1, 0, 0)$ meets the future lightcone of the point $(1, 0, 0)$ not in a single point, as in the 2 dimensional case, but in one branch of the hyperbola $u=0,\; t^2=1+v^2$, so $(-1,0,0), (1,0,0)$ have no join, nor does any unordered pair of points have a join or a meet.

  Returning to the two dimensional case, we also define binary relations (illustrated in figure~\ref{fig:quads})  $<, {\prec,}  {\leq_1,} {\leq_2,}  \triangleleft$ by
\begin{align*}
(x, y)<(x', y')&\iff (x, y)\leq(x', y')\wedge (x, y)\neq (x', y')\\
(x, y)\prec(x', y')&\iff x<x'\wedge y<y'\\
(x, y)\leq_1(x', y')&\iff x\leq x'\wedge y=y'\\
(x, y)\leq_2(x', y')&\iff x=x'\wedge y\leq y'\\
(x, y)\triangleleft(x', y') &\iff x\leq x'\wedge y\geq y'\wedge (x, y)\neq(x', y').
\end{align*}
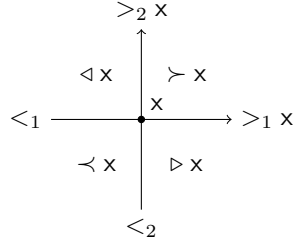
\begin{figure}
\begin{tikzpicture}[scale=.6]
\draw [->] (-2,0) -- (2, 0);
\draw[->] (0, -2) -- (0, 2);
\node [above right] at (0, 0) {$\x$};
\node at (1,1) {$\succ \x$};
\node at  (-1,-1) {$\prec \x$};
\node at (-1, 1) {$\triangleleft \; \x$};
\node at (1,-1) {$\triangleright \;\x$};
\node [right] at (2,0) {$>_1\x$};
\node [above] at (0,2) {$>_2\x$};
\node[below] at (0, -2) {$<_2$};
\node[left] at (-2,0){$<_1$};

\draw [fill] (0,0) circle [radius=2pt];
\end{tikzpicture}
\caption{\label{fig:quads}Binary relations on $\reals^2$}
\end{figure}
  We write $\x, \y, \z$ etc. to denote typical points in $\reals^2$.   We may write $\x<_1\y$ if $\x\leq_1\y$ and $\x\neq \y$ and in this case we may say that $\y$ is due East of $\x$; similarly if $\x<_2\y$ we may say that $\y$ is due North of $\x$.   Observe that $\triangleleft$ contains $<_1$ and $>_2$, and for any $\x, \y\in\reals^2$, exactly one of the following holds: $\x=\y,\; \x\prec \y,\; \y\prec \x,\;\x\triangleleft\y,\;\y\triangleleft\x$.  We write $\x^\uparrow$ for $\set{\y\in\reals^2:\y\geq \x}$ and $\x^\downarrow$ for $\set{\y\in\reals^2:\y\leq\x}$.   If a non-empty subset $S$ of $\reals^2$ is bounded below (respectively above) there is a unique infimum (supremum) of $S$ written $\bigwedge S$ ($\bigvee S$).  For arbitrary $\x, \y\in\reals^2$ we write $\x\vee\y$ for $\bigvee\set{\x, \y}$ and define $\x\wedge\y$ similarly.  Note that $\x\vee\y$ is the greater of $\x$ and $\y$ if they are ordered, else it is the unique solution of $\x\leq_1(\x\vee\y)\geq_2\y$ or $\x\leq_2(\x\vee\y)\geq_1\y$.  
   A \emph{spatial set} $S\subseteq\reals^2$ satisfies $\forall\x, \y\in S \; \neg(\x< y)$.  An \emph{ordered set} $S$ satisfies $\forall\x, \y\in S\;(\x<\y\mbox{ or }\y>\x\mbox{ or }\x=\y)$.    A \emph{line} is a set $\set{(x, y)\in\reals^2:y=mx+c}$ or $\set{(a, y)\in\reals^2}$, for any constants $m, c, a$.  $\set{(x, y):y=mx +c}$ is spatial if $m<0$, other lines are ordered.  For $m>0$ the line $\set{(x, y):y=mx+c}$ is a slower than light line.   A line segment is a convex subset of a line.  

Until theorem~\ref{thm:complexity}, we focus on the validities over the irreflexive frame $(\reals^2, <)$.  Having proved the decidability of these validities, the decidability of the set of validities over the reflexive frame $(\reals^2, \leq)$ easily follows from the following equivalence
\[ (\reals^2, \leq)\models\phi \iff (\reals^2, <)\models\phi'\]
where $\phi'$ is obtained from $\phi$ by replacing each subformula $\F\psi$ of $\phi$ by $(\psi\vee\F\psi)$ and each subformula $\P\theta$ by $(\theta\vee\P\theta)$.\\

\subsection{Axioms}
Let $\Ax$ be the basic modal logic $K$, temporal axiom and duals: it consists of an axiomatisation of propositional logic together with all instances of
\begin{align*}
\G(\psi\rightarrow\theta)&\rightarrow(\G\psi\rightarrow\G\theta)&
\psi&\rightarrow \G\P\psi\\
\H(\psi\rightarrow\theta)&\rightarrow(\H\psi\rightarrow\H\theta)&\psi&\rightarrow\H\F\psi
\end{align*}
with modus ponens and  temporal generalisation as inference rules.  These axioms and rules are clearly sound over arbitrary frames but certainly not complete for $(\reals^2, <)$ (e.g. the transitivity axiom $\G\psi\rightarrow\G\G\psi$ and the confluence axiom $\F\G\psi\rightarrow\G\F\psi$ do not follow from $\Ax$) but they suffice for our construction.

\subsection{Filtration}
Let $\phi$ be a formula,  $Cl(\phi)=\set{\psi, \neg\psi:\psi\mbox{ is a subformula of }\phi}$ is the closure set.  Throughout this paper, $\phi$ will be a fixed formula whose satisfiability we wish to determine, all formulas will be restricted to formulas in $Cl(\phi)$.   \MCS is the set of maximal consistent subsets of $Cl(\phi)$ wrt $\Ax$.  \MCS has an ordering $<$ given by 
\begin{eqnarray*}
m< n &\iff &\mbox{for all }\F\psi\in Cl(\phi),\;  ((\psi\in n\rightarrow \F\psi\in m)\wedge(\F\psi\in n\rightarrow \F\psi\in m))    \\  
&&\wedge\;\;
\mbox{for all }\P\psi\in Cl(\phi), \;   ((\psi\in m\rightarrow \P\psi\in n)\wedge(\P\psi\in m\rightarrow \P\psi\in n)).
\end{eqnarray*}
 where $m, n\in \MCSx$.   Write $m\sim n$ if  $m< n$ and $n< m$.   It follows from this definition that $<$ is transitive, also  by consistency and maximality of $m<n$ if $\G\psi\in m$ then $\psi, \G\psi \in n$ and a dual property for $\H$.  Hence,
  if $m\sim n$ and $\theta\in Cl(\phi)$ is any temporally bound formula (i.e. either $\F\psi, \P\psi, \G\psi$ or $\H\psi$, some $\psi$) then $\theta\in m\iff\theta\in n$.  A \emph{cluster} is a maximal clique $c$ of \MCSss, so $m, n\in c \rightarrow m\sim n$ and if $m'\not\in c$ then $m\not\sim m'$.  The set of clusters forms a partition of the set of reflexive \MCSss.  If $m$ is a reflexive $\MCS$ we write $[m]$ for the cluster containing $m$.     We may write $c\leq d$, for two clusters $c, d$,  when $m\leq n$ for some (equivalently all) $m\in c,\; n\in d$.  We write $c<d$ to mean $c\leq d$ and $c\neq d$, this defines an irreflexive ordering of clusters.  We can extend the use of these relations $<, \leq$ to relate $\MCSs$  to clusters, thus $m\leq c$ means that for all $m'\in c$ we have $m\leq m'$ and $m<c$ means for all $m'\in c$ we have $m<m'$.

 For $\psi\in Cl(\phi)$ and a cluster $c$ of $\MCSs$  we may say that ``$\phi$ belongs to $c$'' if $\phi\in\bigcup c$, i.e. if there is an $\MCS$ $m$ such that $\phi\in m\in c$.  

\begin{definition}\label{def:trace}
\begin{itemize}
\item A \emph{trace} is an alternating ordered sequence of clusters and $\MCSs$,\break $(c_0, m_0, c_1, m_1, \ldots, m_{k-1}, c_k)$ (some $k\geq 0$) where $c_i\leq m_i\leq c_{i+1}$ and $c_i<c_{i+1}$, for $i<k$.  
\item  $c_0$ is the initial cluster of the trace and $c_k$ is the final cluster.  
\item Given two traces $\trace=(c_0, m_0, \ldots, c_k)$ and $\trace'=(c'_0, m'_0,\ldots, c'_{k'})$ and an $\MCS$ $m$ such that $c_k\leq m\leq c'_0$, we let  $\trace+_m\trace'$ denote the trace
 \begin{align*}
(c_0, m_0, \ldots, m_{k-1}, c_k, m'_0, c'_1, \ldots, c'_{k'})& \mbox{if } c_k=c'_0, \mbox{ or}\\
(c_0, m_0, \ldots, c_k, m, c'_0, m'_0, c'_1, \ldots, c'_{k'})& \mbox{if }  c_k<c'_0
\end{align*}
 so in the former case the the final and initial clusters are combined into one, in the latter they are separated by $m$.  

\item We may write $\trace\leq c$ to mean that all clusters and  \MCSs  of $\trace$ are $\leq c$, etc.    

\item An \emph{extended trace} $(m_{-1}, c_0, \ldots, , c_k, m_k)$ (some $k\geq 0$) consists of a trace $(c_0, \ldots, c_k)$ preceded by an \MCS $m_{-1}\leq c_0$ and followed by an \MCS $m_k\geq c_k$.
\end{itemize}
\end{definition}

For any $\x\in\reals^2$ and any 2D structure $\c M$, let $\x^\c M$ denote the \MCS of all formulas in $Cl(\phi)$ that hold at $\x$.  Observe, by definition of our semantics, that 
\begin{equation}\label{eq:M}
\x<\y\rightarrow \x^\c M\leq \y^\c M \mbox{ and }  \F\psi\in \x^\c M \leftrightarrow\exists \y> \x\; \psi\in \y^\c M  \end{equation} provided $\F\psi\in Cl(\phi)$, and a dual property for $\P$.  Let $l\subset \reals^2$ be any ordered  line or ordered open line segment and let $\lambda:(0, 1)\rightarrow l$ be a  continuous order preserving surjection.  The map $u\rightarrow \lambda(u)^\c M$ is an order preserving
 function from $(0, 1)$ to \MCSss.   For each $u\in(0, 1)$, \/ $\lambda(u)^\c M$ is either an irreflexive \MCS or it belongs to a cluster.  Since there are only finitely many \MCSs and clusters,
  $l$ determines a trace
 $l^{\c M} = (c_0, m_0, \ldots, c_k)$  such that there are points $\x_0, \x_1, \ldots, \x_{i-1}\in l$ where  $\x_i^\c M=m_i$ and for all $\y\in l$ if $\y<\x_0$ then $\y^\c M\in c_0$, if $\x_i<\y<\x_{i+1}$ then $\y^\c M\in c_{i+1}$, for $i<k$, and if $\y>\x_{m-1}$ then $\y^\c M\in c_k$.  
 Similarly, a closed, proper line segment determines an extended trace.

\subsection{Defects}
For any \MCS $m$ and future formula $\F\psi$, if $\F\psi\in m$ then $\F\psi$ is a defect of $m$.  
$\F\psi$ is a defect of the cluster  $c$ if $\F\psi$ belongs to $c$ but $\psi$ does not belong to $c$.    Let $a\leq  b$ be either clusters or \MCSs and suppose $\F\psi$ is a defect of $a$.  If either $\psi$ or $\F\psi$ belongs to $b$ then we say the defect $\F\psi$ is passed up to $b$ from $a$.
If $a$ is a cluster or \MCS of a trace $\trace$, $\F\psi$ is a defect of $a$ and either $a$ is the final cluster of $\trace$ or $\F\psi$ is not passed up to the next  cluster or \MCS of $\trace$ then $\F\psi$ is a defect of $\trace$.  
Similarly we may define  $\P$ defects of \MCSss, clusters and traces.

\subsection{Rectangle models}\label{sec:rectangles}
\begin{definition}
A \emph{rectangle} is a non-empty set $D\subseteq\reals^2$  satisfying  
\[(\x, \y\in D\&((\x\wedge\y)\leq\z\leq(\x\vee\y)))\rightarrow\z\in D.\]
\end{definition}
A rectangle is \emph{proper} if it has non-empty interior.
The whole plane $\reals^2$ is a key example of an unbounded  open, proper rectangle.    The two projections  $\set{x\in\reals:\exists y (x, y)\in D}$ and $\set{y\in\reals:\exists x (x, y)\in D}$ of a rectangle $D$  define  non-empty convex segments in the reals, furthermore $D$ is the Cartesian product of the two projections.   Observe that the set of  convex segments of real numbers may be divided into five categories, according to whether or not 
the segment includes its infimum and supremum, and (if so) whether the infimum and supremum are distinct: segments not containing either their infimum or supremum (open), segments containing their infimum but not their supremum (right open, left closed and bounded), segments containing their supremum but not their infimum (left open, right closed and bounded), proper segments containing their infimum and their supremum (closed and bounded) and one point segments.  Thus there are 25 kinds of bounded rectangles, examples are shown in figure~\ref{fig:rectangles}: 16 are proper (one of these is open, four have a single boundary edge, six have two boundary edges, four have three boundary edges and one is a closed proper rectangle), 8 consist of vertical or horizontal proper bounded segments and 1 is a one point rectangle.   

  For $(x, y)\in\reals^2$ we write $W(x, y)$ for the half-plane rectangle $\set{(x', y')\in\reals^2: x'\leq x}$ and define $E(x, y),\; N(x, y)$ and $S(x, y)$ similarly.  
  For any $\x\prec\y\in\reals^2$ we write $[\x, \y]$ for the proper closed rectangle $\set{\z\in\reals^2:\x\leq\z\leq\y}$. 
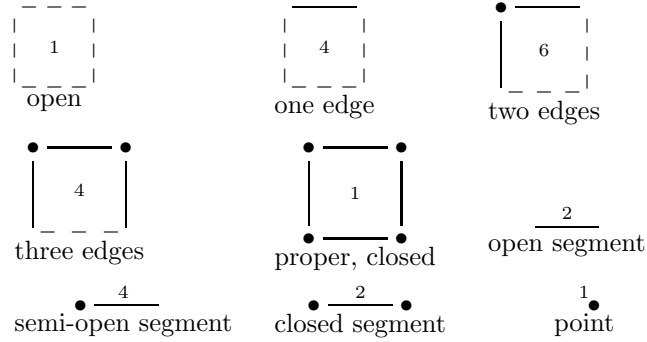
\begin{figure}
\[
\begin{array}{lll}
\xymatrix{
\ar@{}[rd]|{1}\ar@{--}[r]\ar@{--}[d]&\ar@{--}[d]\\
\ar@{--}[r]_{\mbox{open}}&
}
&
\xymatrix{
\ar@{}[rd]|{4}\ar@{-}[r]\ar@{--}[d]&\ar@{--}[d]\\
\ar@{--}[r]_{\mbox{one edge}}&
}
&\xymatrix{
\bullet\ar@{}[rd]|{6}\ar@{-}[r]\ar@{-}[d]&\ar@{--}[d]\\
\ar@{--}[r]_{\mbox{two edges}}&
}\\
\xymatrix{
\bullet\ar@{}[rd]|{4}\ar@{-}[r]\ar@{-}[d]&\bullet\ar@{-}[d]\\
\ar@{--}[r]_{\mbox{three edges}}&
}
&\xymatrix{
\bullet\ar@{}[rd]|{1}\ar@{-}[r]\ar@{-}[d]&\bullet\ar@{-}[d]\\
\bullet\ar@{-}[r]_{\mbox{proper, closed}}&\bullet
}
&
\xymatrix{\\
\ar@{-}[r]^2_{\mbox{open segment}}&
}
\\
\xymatrix{
\bullet\ar@{-}[r]^4_{\mbox{semi-open segment}}&
}
&
\xymatrix{
\bullet\ar@{-}[r]^2_{\mbox{closed segment}}&\bullet
}
&

\xymatrix{
\ar@{}[r]^{\;\;\;\;\;\;\;\;\;\;\;1}_{\;\;\;\;\;\;\;\;\mbox{ \/\/\/\/ point}}&\bullet
}
\end{array}
\]
\caption{\label{fig:rectangles} 25 different types of rectangles of which 16 are proper.  Segments can be horizontal (shown) or vertical.  }
\end{figure}
If $\set{\x, \y}$ are incomparable (i.e. $\x\not\leq\y$ and $\y\not\leq\x$) then $[\x\wedge\y, \x\vee\y]$ is a proper closed rectangle with $\x, \y$ at the left and right corners (either way round), any point in the interior of $[\x\wedge\y, \x\vee\y]$ is said to be \emph{between} $\x$ and $\y$.  The following statement is valid in $(\reals^2, \leq)$ but does not follow from {\bf Ax}: if $\x, \y, \z$ are pairwise incomparable then either $\x$ is between $\y$ and $\z$, \/ $\y$ is between $\x$ and $\z$ or $\z$ is between $\x$ and $\y$.  A spatial set $S$ is said to be densely partitioned by $\set{S_1, S_2}$ if it is the disjoint union of $S_1$ and $S_2$ and for all distinct points $\x, \y\in S$ there are $\z_1\in S_1,\; \z_2\in S_2$ between $\x$ and $\y$.

The following lemma will be used extensively, we omit the proof which  is quite obvious.
\begin{lemma}\label{lem:equiv}\label{lem:path}\label{lem:order}  
Let $f_1:i_1\rightarrow j_1$ be an order preserving bijection from the segment $i_1\subseteq\reals$ onto the segment $j_1\subseteq\reals$ and let $f_2:i_2\rightarrow j_2$ be another order preserving bijection from one segment onto another.  The map $(f_1, f_2): (i_1\times i_2) \rightarrow (j_1\times j_2)$ defined by $(f_1, f_2)(x, y)=(f_1(x), f_2(y))$, for $x\in i_1, \; y\in i_2$,  is an order preserving bijection from the rectangle $i_1\times i_2$ onto the rectangle $j_1\times j_2$.
\end{lemma}

\begin{definition}\label{def:rect}
A \emph{rectangle model} $h:D\rightarrow \MCS$ has a proper rectangle $D\subseteq\reals^2$ as its domain and satisfies
\begin{description}
\item[coherence]
$\x\leq\y\in D\rightarrow h(\x)\leq h(\y)$ and
\item [no internal defects]
If $\F\psi\in h(\x)$ then either there is $\y\in D$ with $\y>\x$ and $\psi\in h(\y)$ (no defect) or
\begin{itemize}
\item  $D$ includes the boundary point  $\y$ due East of $\x$ and $\F\psi\in h(\y)$ (the defect may be passed East to a neighbouring rectangle model) or
\item  $D$ includes the boundary point $\y$ due North of $\x$ and $\F\psi\in h(\y)$ (the defect may be passed North).
\end{itemize}
and (similarly) where all occurrences of $\P$ defects may be passed South or West.  
\end{description}
   If for all $\x<\y\in D$ such that $h(\x)\sim h(\y)$, for all $m\sim h(\x)$ there is $\z\in D$ with $\x<\z<\y$ and $h(\z)=m$, then we say that $h$ maps \emph{densely}.
\end{definition}
\begin{lemma}\label{lem:truth}
Let $h$ be a rectangle model with $\dom(h)=\reals^2$.  Let $v_h:\props\rightarrow\wp(\reals^2)$ be the valuation defined by $v_h(p)=\set{\x\in\reals^2: p\in h(\x)}$, for $p\in\props$.  Then
\begin{equation}\label{eq:truth} (\reals^2, <, v_h), \x\models\psi\iff \psi\in h(\x)
\end{equation}
for $\psi\in Cl(\phi)$ and $\x\in\reals^2$.

Conversely, if $v:\props\rightarrow\wp(\reals^2)$ is any valuation, then the map $h_v:\reals^2\rightarrow \MCS$ defined by $h_v(\x)=\set{(\psi\in Cl(\phi): (\reals^2, <, v), \x\models\psi}$ is an open rectangle model.
\end{lemma}
\begin{proof}
For the first part, note that if $h$ is a rectangle model and $\dom(h)$ is open then $h$ has no defects, so we may prove \eqref{eq:truth} by structured formula induction.  The second part follows from \eqref{eq:M} and its dual.
\end{proof}

\begin{definition}\label{def:trace2}
Let $h$ be a rectangle model and let $l\subseteq \dom(h)$ be an ordered, open line segment.   Since there are only finitely many \MCSss, there is a trace $(c_0, m_0, \ldots, c_k)$ and points $\x_0, \ldots, \x_{k-1}\in l$ such that $h(\x_i)=m_i$ and for $\y\in l$ if $\x_i<\y<\x_{i+1}$ then $h(\y)\in c_{i+1}$ (all $i<k$), if $\y<\x_0$ then $h(\y)\in c_0$ and if $\y>\x_{k-1}$ then $h(\y)\in c_k$.  We denote the trace by $\trace(h, l)$ and we denote the point $\x_i$ by $\x(h, l, i)$.
\end{definition}

For any open rectangle model $h$, let $\x\in \dom(h)$ be any point such that for all $\y\leq \x\in \dom(h)$ we have $h(\y)\sim h(\x)$ ($\x$ must exist since there are only finitely many \MCSss).  Then $\set{\y\in \dom(h): \y\prec \x}$ is an open rectangle labelled by a single cluster  $[h(\x)]$, less than or equal to all \MCSs in the range of $h$.  We call $[h(\x)]$ the \emph{lower cluster} of $h$, similarly there is an \emph{upper cluster}  that holds over an open rectangle at the top of $\dom(h)$.   For an arbitrary rectangle model $h$, the upper cluster (respectively lower cluster) of $h$ is the upper cluster (lower cluster) of the restriction of $h$ to the open interior of its domain.  The \emph{height} of a rectangle model is the maximum possible length of a chain of distinct clusters from the lower cluster up to the upper cluster.  
Let $\b\prec \t\in\reals^2$, so $[\b, \t]$ is a closed proper rectangle.  A closed proper  rectangle model $h$ over $[\b, \t]$ has four corners: $\b, \t, \l, \r$ where $\b<_2 \l <_1 \t$ and $\b<_1\r<_2\t$, see figure~\ref{fig:rectangle}(a).  From these corners we get four \MCSss: bottom $=h(\b)$, top $=h(\t)$, left $=h(\l)$ and right $=h(\r)$.  Also, four traces: one from each of the four open segments of the boundary.  
A point model is a closed rectangle model over singleton rectangle $[\x, \x]$.  It's labelled by just a single \MCS 
$h(\x)$,  it has empty interior.

\subsection{Topology}

\begin{lemma}\label{lem:point}
If  $S$ is a partition of an open segment into closed segments (segments containing both endpoints, possibly equal) then $S$ contains an uncountable set of singleton point, closed segments.    If $S$ is a partition of a proper closed segment into more than one closed segments,  then $S$  contains uncountably many singleton points.
\end{lemma}
\begin{proof}   For each function $\lambda:\omega\rightarrow\set{0, 1}$  other than the constant $0$ function and the constant $1$ function, we will construct  a singleton element $[x_\lambda, x_\lambda]\in S$.  We define $x_\lambda$ as the limit of a Cauchy sequence
$a_0, a_1, \ldots$, as follows.  Without loss, we consider a partition of the open segment $(0, 1)$ into closed segments, so each $x\in (0, 1)$ is covered by a unique closed segment, say $x\in [l(x), u(x)]\in S$.    If $\lambda(0)= 0$ we let $a_0=\frac13$  and if $\lambda(0)=1$ we let $a_0=\frac23$.   The gap between $0$ and $l(a_0)$ is at most $\frac23$, also the gap between $1$ and $u(a_0)$ is at most $\frac23$.      Suppose we have already defined a sequence $a_0, a_1, \ldots, a_k$ (some $k\geq 0$).       Let $a_k^+>a_k$ be the minimum of  $\set{1, l(a_j):j<k,\;  l(a_j)>u(a_k)}$ and let $a_k^-<a_k$ be the maximum of  $\set{0, u(a_j):j<k,\; u(a_j)<l(a_k)}$.  Suppose inductively that $a_k^+-a_k\leq\frac23\cdot \frac12^k$ and $a_k-a_k^-\leq\frac23\cdot\frac12^k$.  If $\lambda(k+1)=1$ we let $a_{k+1}$ be the midpoint of $u(a_k)$ and $a_k^+$, else  (when $\lambda(k+1)=0$) let $a_{k+1}$ be the midpoint of $l(a_k)$ and $a_k^-$.  The inductive condition is maintained by this.   This defines our sequence $a_0, a_1, \ldots$.

Observe, when $\lambda(k+1)=1$, that for all $j>k$ we have $u(a_k)<a_j<a_k^+$ and when $\lambda(k+1)=0$ then for all $j>k$ we have $a_k^-<a_j<l(a_j)$.  Hence the sequence $a_0, a_1, \ldots$ is a Cauchy sequence, let the sequence converge to $x_\lambda\in [0, 1]$.  If $\lambda$ is not constantly zero we get $x_\lambda>0$ and if it is not constantly one we get $x_\lambda<1$.  For any non-constant function $\lambda:\omega\rightarrow\set{0, 1}$, by our assumption, $x_\lambda\in[l(x_\lambda), u(x_\lambda)]\in S$.  By construction of the sequence, there are closed segments arbitrarily close to $x_\lambda$ above and below,  and since the segments are disjoint, the length of the segment $u(x_\lambda)-l(x_\lambda)$ is zero and $[l(x_\lambda), u(x_\lambda)]\in S$ must be the singleton closed segment $[x_\lambda, x_\lambda]$.  There are uncountably many non-constant sequences $\lambda:\omega\rightarrow\set{0, 1}$ and clearly, distinct functions give rise to distinct sequences with distinct limits.  Hence there are uncountably many singletons in $S$.

For the second statement, let $i_0 < i_1$ be any two closed segments in a partition of $[0, 1]$ into closed segments.  Then the partition contains a subset that partitions the open segment between $i_0$ and $i_1$ into closed segments, so by the first part it must include  uncountably many singleton points.
\end{proof}

\medskip

The closure $\overline S$ of any subset  $S$ of $\reals^2$ is the intersection of all closed sets containing $S$ under the usual topology for $\reals^2$, the interior  of a set  $S$ is the union of all open sets contained in the set, denoted $Int(S)$.  The boundary of $S$ is the closure minus the interior (note that the boundary is closed).   A point  $\x$ belongs to the boundary of $S$ iff every neighbourhood of $\x$ contains points in $S$ and also points outside $S$.  If $S$ is upward directed and bounded above then $\bigvee S\in\overline{S}$.

\begin{lemma}\label{lem:cont}
Let $S$ be a downward closed subset of the closed rectangle $D=[0, 1]\times[0, 1]$.    Let $l$ be an ordered  line.  If $l$ meets $S$ and $l$ also meets  $D\setminus S$ then $l$ meets the boundary of $S$ in $D$.  If $l$ is a slower than light line then the intersection of $l$ with the boundary of $S$ in $D$ is  a single point.  
\end{lemma}
\begin{proof}
The set $S\cap l$ is a Dedekind cut in the line segment $D\cap l$, let $\x$ be the supremum of $S\cap l$ ($\x\in D\cap l$ since $D\cap l$ is closed and directed).  Every neighbourhood of $\x$ meets $S$ and  $D\setminus S$, hence $\x$ is in the boundary of $S$ in $D$.   If $l$ is a slower than light line then every point in $D\cap l$ above  $\x$ is contained in a neighbourhood disjoint from $S\cap l$ and  every point on $l$ below $\x$ is in the interior of $S$, hence the intersection of $l$ with the boundary of $S$  in $D$ consists of $\x$ only.
\end{proof}

\begin{lemma}\label{lem:segment}
Let $D=[0, 1]\times[0, 1]$.  Let $S\subset D$ be closed downward within $D$ such that $S, D\setminus S$ are non-empty.  The boundary $\Gamma$ of $S$ in $D$ is homeomorphic to a closed line segment or a single point.
\end{lemma}
\begin{proof}  Since $S, D\setminus S$  are non-empty, $\Gamma$ is non-empty.   For $u\in\reals$ the line $y=x+u$ is a slower than light line.  Define a partial  function $f:\reals \rightarrow \Gamma$ by letting $f(u)$ be the unique point in $\Gamma$ on the line $y=x+u$ if the line meets $\Gamma$, undefined otherwise, this is well defined by the previous lemma.  The range of $f$ is $\Gamma$ and the domain is a subset of $[-1,1]$.
If $u, v\in \dom(f)$ and $w$ is between $u$ and $v$  then since $S$ is downward closed, the line $y=x+w$ meets $S$ in  $(f(u)\wedge f(v))^\downarrow$  and meets  $D\setminus S$ in  $(f(u)\vee f(v))^\uparrow$  so by the previous lemma the line meets $\Gamma$ and $w\in \dom(f)$, so the domain of $f$ is convex. 
Furthermore,  for $u<v \in \dom(f)$ we have  $v-u\leq |f(v)-f(u)|\leq\sqrt2(v-u)$.  Hence $f$ is continuous, injective and has a continuous inverse.  The inverse of $f$ maps   $\Gamma$ into the reals, and $\Gamma$ is closed and bounded, hence compact.  It follows that  the range of $f^{-1}$ (i.e. the domain of $f$) is also compact so by  the Heine-Borel theorem it is closed and bounded. Thus $\dom(f)$ is a convex closed bounded subset of $\reals$, hence  a bounded closed segment or a single point.  Thus, $f$ is  the required homeomorphism.
\end{proof}

\begin{lemma}\label{lem:apart}
Let $\Gamma, \Gamma'$ be compact subsets of a metric space $(X, d)$.  Suppose for all $\epsilon>0$ there are $\x\in\Gamma$ and $\x'\in\Gamma'$ such that $d(\x, \x')<\epsilon$ then $\Gamma\cap\Gamma'\neq\emptyset$.
\end{lemma}
\begin{proof}
The metric $d$ is a continuous function from the compact space $\Gamma\times\Gamma'$ to non-negative reals.  By the extreme value theorem this function attains its infimum, i.e. there are $\x\in\Gamma,\;\x'\in\Gamma'$ such that for all $\y\in\Gamma$ and $\y'\in\Gamma'$ we have $d(\x,\x')\leq d(\y, \y')$.  By the supposition in the lemma, $d(\x,\x')<\epsilon$ for all $\epsilon>0$, hence $d(\x, \x')=0$ and $\x=\x'\in\Gamma\cap\Gamma'$.
\end{proof}

\medskip

For $X\subseteq\reals^2$, the \emph{upper boundary} of $X$ consists of those points $\x$ in the boundary of $X$ such that $\y\succ\x\rightarrow \y\not\in X$ and the lower boundary consists of those $\x$ in the boundary such that $\y\prec\x\rightarrow\y\not\in X$.  For example, if $X$ is a  rectangle then the upper boundary consists of the edges labelled $N, E$  and the corners labelled $\l, \r, \t$ in figure~\ref{fig:rectangle}(a) , while the lower boundary consists of the edges labelled $S$ and $W$ and the corners labelled $\b, \l, \r$.

\section{Boundary Maps}\label{sec:BM}
Throughout this paper, given two partial functions $f, g$ we write $f(x)=g(x)$ to mean that $f(x), g(x)$ are either both defined and equal or both undefined.

\begin{definition}\label{def:bm} See figure~\ref{fig:rectangle}(a).
A \emph{ proper boundary map} $\partial$ is  a partial map from $\set{N, S, E, W} \cup \set{\t, \l, \r, \b} \cup\set{-, +}$  to traces (for the first four listed), \MCSs (for the next four) or clusters (for the last two) such that $\set{-, +}\subseteq \dom(\partial)$  and: 
\begin{itemize}
\item  every future defect of $\partial(+)$ must be passed up to the final cluster of $\partial(N)$ (which must be defined in this case) or the final cluster of $\partial(E)$ (which must be defined in this case),
\item  if $N\in \dom(\partial)$ then $\partial(+)$ is less than or equal to the 
final cluster of $\partial(N)$,
\item  $\t\in \dom(\partial) \iff \set{N, E}\subseteq \dom(\partial)$ and if this holds $\partial(N), \partial(E)\leq\partial(\t)$ and all $\P$ defects of $\partial(\t)$ are passed down to 
either the final cluster of $\partial(N)$ or the final cluster of $\partial(E)$,
\item $\l\in\dom(\partial)\iff \set{W, N}\subseteq\dom(\partial)$ and if this holds $\partial(W)\leq\partial(\l)\leq\partial(N)$,  
\item dual properties obtained by swapping $(\F, \P), (\t,\b), (N, W), (S, E)$, 
$(+, -),$ $($up/down$)$ throughout, or/and by swapping $(\l, \r), (N, E), (S, W)$ throughout.
\end{itemize}
A \emph{one point boundary map} $\partial$ is a constant map from $\set{\t, \l, \r, \b}$ to \MCSx, i.e. $\partial(\t)=\partial(\l)=\partial(\r)=\partial(\b)\in \MCSx$.  For any $m\in \MCS$ let $\partial_m$ denote the one point boundary map defined by $\partial_m(\b)=m$.
\end{definition}
It follows that the domain of a proper boundary map  (ignoring $-, +$) corresponds to a boundary of one of the 16 proper rectangles (see figure~\ref{fig:rectangles}).  
A proper boundary map  $\partial$ is \emph{open} if $\dom(\partial)=\set{-, +}$, it is \emph{closed} if $\dom(\partial)=\set{N, S, E, W, \t, \l, \r, \b, -, +}$.   One point boundary maps are  closed.  The \emph{height} of a proper boundary map $\partial$ is the maximum length of a chain of distinct clusters from $\partial(-)$ up to $\partial(+)$.  

A \emph{past defect} of a proper boundary map $\partial$ is either a past defect of $\partial(\b)$, a past defect of  either $\partial(S)$ or $\partial(W)$ not passed down to $\partial(\b)$, a past defect of $\partial(\l)$ not passed down to the final cluster of $\partial(W)$, or a past defect of $\partial(\r)$ not passed down to the final cluster of $\partial(S)$ (in each case, only if the relevant MSC, cluster or trace is defined, of course).  A past defect of a one point boundary map $\partial_m$ is a past defect of $m$.   A \emph{future defect} of a boundary map is defined similarly.  
\begin{definition}
Let $h$ be a rectangle model.  Define a proper boundary map $\partial^h$ as follows.  $\partial^h(-)$ is the lower cluster of $h$, \/ $\partial^h(+)$ is the upper cluster.   If $\dom(h)$ includes a point $\x$ at one of its four corners $c$ then $\partial^h(c)=h(\x)$, otherwise $\partial^h(c)$ is not defined,  (e.g. if $\dom(h)$ includes its infimum $\x$ then $\partial^h(\b)=h(\x)$),\/ if $\dom(h)$ includes one of its edges,    say the edge is the open line segment $l\subseteq \dom(h)$ located on the edge in the direction $d\in\set{N, S, E, W}$ of the rectangle, then $\partial_h(d)=h(l)$, otherwise it is undefined.

Let $h$ be a rectangle model and let $R$ be any rectangle such that $\dom(h)\cap R$ is a proper rectangle.  We write $\partial(h, R)$ for the boundary map $\partial^{h\restr{R}}$.   If $\x\in \dom(h)$ we let $\partial(h, \set\x)$ be the one point boundary map $\partial_{h(\x)}$.
\end{definition}
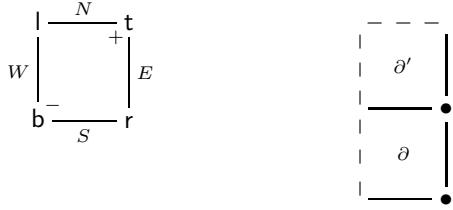
\begin{figure}
\[
\xymatrix{
\l\ar@{-}[r]^N&\t\ar@{-}[d]^E\\
\b\ar@{-}[u]^W\ar@{-}[r]_S\ar@{}[ru]|<{-}|>{+}&\r
}\hspace{1in}
\xymatrix{
\ar@{--}[r] \ar@{--}[d]  \ar@{}[dr]|{\partial' } &    \ar@{-}[d] \ar@{-}[d] \ar@{-}[d]\\
\ar@{--}[d] \ar@{-}[r]\ar@{-}[r]  \ar@{}[dr]|{\partial }  &   \bullet \ar@{-}[d]\ar@{-}[d]   \\
\ar@{-}[r] \ar@{-}[r]       &    \bullet
}
\]
\caption{\label{fig:rectangle} (a) A closed rectangle with its boundaries shown, (b) \label{fig:match} $\partial'$ fits to the North of $\partial$.  Dashed lines represent open edges, solid lines represent closed edges.  Here $\dom(\partial')=\set{S, E, \r, -, +},\; \dom(\partial)=\set{S, E, N, \r, \t, -, +}$. }
\end{figure}
We have three ways of building bigger proper boundary maps out of smaller ones.
\begin{definition}\label{def:combs}
\begin{description}
\item[Joins]
Given two proper boundary maps $\partial, \partial'$ we say $\partial'$ fits to the North of $\partial$ if  $\partial(N)$ is defined, \/ $\partial(N)=\partial'(S)$, \/ $\partial(\l)=\partial'(\b)$ and $\partial(\t)=\partial'(\r)$.  An example  is shown in figure~\ref{fig:match}(b).   When this happens we write $(\partial \oplus_{N}\partial')$ for the proper boundary map defined by
\begin{itemize}
\item $(\partial \oplus_{N}\partial')(\b)=\partial(\b),\; (\partial \oplus_{N}\partial')(\t)=\partial'(\t),\; (\partial \oplus_{N}\partial')(\l)=\partial'(\l)$ and $(\partial \oplus_{N}\partial')(\r)=\partial(\r)$,
\item $(\partial \oplus_{N}\partial')(S)=\partial(S),\; (\partial \oplus_{N}\partial')(W)=\partial(W)+_{\partial(l)}\partial'(W)$, with similar definitions of the other two boundary traces.
\end{itemize}
Following our convention for partial functions the definition  $(\partial \oplus_{N}\partial')(\b)=\partial(\b)$ makes the left hand side undefined if $\partial(\b)$ is undefined.  
We have similar definitions for joining two proper boundary maps in the other three directions.  
%

\item[Limits]
  We write $\partial \stackrel{W}{\rightarrow}\partial^*$  (see figure~\ref{fig:limit}) if 
\begin{itemize}

\item  $\partial$ fits to the West of $\partial$ (hence $\partial=\partial\oplus_W\partial$, and if $\partial(S)$ is defined it consists of a single cluster, similarly for $\partial(N)$),
\item $\partial^*$ is identical to $\partial$ except on $W, \b, \l$.   If defined, $\partial^*(W)\leq \partial^*(-)$ (i.e. all \MCSs and clusters of $\partial^*(W)$ are $\leq\partial^*(-)$) and every future defect of $\partial^*(W)$ is either passed up to $\partial^*(\l)$ or  $\partial^*(-)$.
\end{itemize}
Again, by our convention for partial functions, in the last line we mean that  if $\partial^*(W)$ is defined then any future defect is either passed up to $\partial^*(\l)$, which must be defined in this case, or it is passed up to $\partial^*(-)$.
We have similar definitions of $\partial\stackrel{d}{\rightarrow}\partial^*$, for $d\in\set{N, S, E}$.  
\begin{figure}  
\begin{center}
\newcommand\Square[1]{+(-#1,-#1) rectangle +(#1,#1)}
\begin{tikzpicture}[scale=.8]
\draw(0,2) \Square{2};

\draw [fill] (2,1) circle [radius=2.0pt];
\draw [fill] (2,3) circle [radius=2.0pt];
\draw [fill] (-2,1) circle [radius=2.0pt];
\draw [fill] (-2,3) circle [radius=2.0pt];

\node[left] at (-2, 1){$\w_0$};
\node[left] at (-2, 3){$\w_1$};
\node[right] at (2,1){$\e_0$};
\node[right] at (2,3){$\e_1$};

\node [right] at (-2,1) {$m_0$};
\node [right] at (-2,3) {$m_1$};
\node [left] at (-2.1,2) {$c_1$};
\node [left] at (-2.1,.5) {$c_0$};
\node [left] at (-2.1,3.5) {$c_2$};
\node[right] at (-2, 2) {$\partial(W)$};
\node[above]  at (0,4) {$\partial(N)$};
\node[below]  at (0,0) {$\partial(S)$};
\node[right]  at (2,2) {$\partial(E)$};
\node [left] at (2,1) {$m_0$};
\node [left] at (2,3) {$m_1$};
\node [left] at (2,2) {$c_1$};
\node [left] at (2,.5) {$c_0$};
\node [right] at (-2,.5) {$=\partial(-)$};
\node [left] at (2,3.5) {$c_2$};
\node [right] at (-2,3.5) {$=\partial(+)$};
\node at (0,2){$\partial$};

\end{tikzpicture}
\begin{tikzpicture}[scale=.8]
\draw(8,2)\Square{2};

\draw [fill] (10,1) circle [radius=2.0pt];
\draw [fill] (10,3) circle [radius=2.0pt];
\draw [fill] (8,3.5) circle [radius=2.0pt];
\draw [fill] (8, 2.5) circle [radius=2.0pt];
\draw [fill] (7, 3.75) circle [radius=2.0pt];
\draw [fill] (7,3.25) circle [radius=2.0pt];
\draw [fill] (6.5, 3.885) circle [radius=2.0pt];
\draw [fill] (6.5,3.625) circle [radius=2.0pt];

\draw [fill] (6.25,3.9375) circle [radius=2.0pt];
\draw [fill] (6.25,3.8125) circle [radius=2.0pt];

\draw [fill] (6.125,3.90626) circle [radius=2.0pt];
\draw [fill] (6.125,3.96875) circle [radius=2.0pt];

\draw[ultra thick](6,0)--(6,4);

\draw (8,0)--(8,4);
\draw (7,0)--(7,4);
\draw (6.5,0)--(6.5,4);
\draw (6.25,0)--(6.25,4);
\draw (6.125,0)--(6.125,4);

\node at (9,2) {$\partial$};
\node at (7.5,2) {$\partial$};
\node at (6.75,2) {$\partial$};

\node[above]  at (8,4) {$\partial^*(N)=\partial(N)$};
\node[below]  at (8,0) {$\partial^*(S)=\partial(S)$};
\node[right]  at (10,2) {$\partial^*(E)=\partial(E)$};
\node[left]  at (6,2) {$\partial^*(W)$};

\node [left] at (10,1) {$m_0$};
\node [left] at (10,3) {$m_1$};
\node [left] at (10,2) {$c_1$};
\node [left] at (10,.5) {$c_0$};
\node [left] at (10.3,3.7) {$\partial^*(+)=c_2$};

\node [overlay] at (7.5,.5) {$\partial^*(-)=\partial(-)$};

\end{tikzpicture}
\end{center}
\caption{\label{fig:limit} $\partial\stackrel{W}{\rightarrow}\partial^*$ and $\partial^*$ is a Western limit of $\partial$.. }
\end{figure}
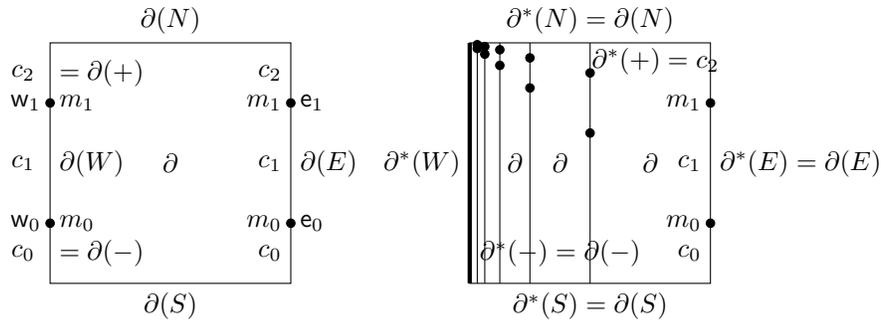

\item[Shuffles]
Suppose $\partial(-)<\partial(+)$,\/ $\set{\partial_i:i\in I}$ is a set of closed boundary maps including at least one one point boundary map, and $\partial$ is a proper boundary map.  We say that $\partial$ is a shuffle of $\set{\partial_i:i\in I}$  if
\begin{enumerate}
\item  for all $i\in I$ we have  $\partial(-)\leq \partial_i(\b)\leq\partial_i(\t)\leq\partial(+)$, \label{sh:1}
\item
 every future defect of $\partial(-)$ is passed up to  $\partial_i(\b)$ for some $i\in I$,\label{sh:2}  
\item   every past defect of $\partial_i$ is passed down to $\partial(-)$ (for all $i\in I$),\label{sh:3}
\item  if defined, $\partial(S)\leq\partial(-)$ and every  future defect of $\partial(S)$ is passed up to $\partial(\r)$  or  $\partial(-)$; 
if defined, $\partial(W)\leq\partial(-)$  and every  future defect of $\partial(W)$ is passed up to $\partial(\l)$ or $\partial(-)$, \label{sh:4}
\item dual properties obtained by swapping the pairs \\
$(\leq, \geq), (-, +), (N, W), (S, E), (\b, \t), (\mbox{up}, \mbox{down}), \; (\F, \P)$, throughout,\label{sh:6}
\end{enumerate}
see figure~\ref{fig:case2} for an illustration of a closed boundary map of this form.  
 
\end{description}
\end{definition}

The following lemma is quite straightforward and we omit the proof.  Recall from section~\ref{sec:rectangles} that $N(\x)$ is the half-plane of points no further South than $\x$.
\begin{lemma}\label{lem:split}
Let $h$ be a rectangle model and let $\x\in Int(\dom(h))$.  We have
\[ \partial^h=\partial(h, S(\x))\oplus_N\partial(h, N(\x))\]
and three similar decompositions (using $\oplus_S, \oplus_E, \oplus_W$) also hold.
\end{lemma}
   
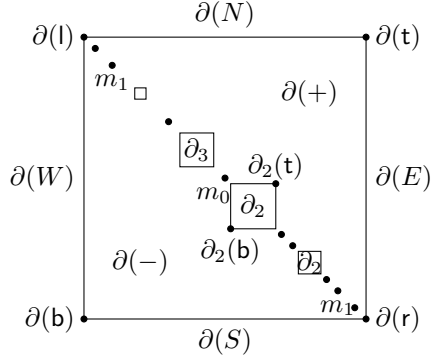
\begin{figure}
\begin{center}
\begin{tikzpicture}[scale=.75]
\newcommand\Square[1]{+(-#1,-#1) rectangle +(#1,#1)}
\draw (0,0) -- (5,0) --(5,5) --(0,5) -- (0, 0);
\node[above] at (2.5, 5) {$\partial(N)$};
\node[below] at (2.5, 0){$\partial(S)$};
\node[left] at (0, 2.5){$\partial(W)$};
\node[right] at (5, 2.5){$\partial(E)$};
\node[above, left] at (0, 5){$\partial(\l)$};
\node[above, right] at (5, 5){$\partial(\t)$};
\node[below, left] at (0, 0){$\partial(\b)$};
\node[below, right] at (5, 0){$\partial(\r)$};
\draw (3, 2) \Square{0.4};
\node at (3, 2) {$\partial_2$};
\draw (2, 3)\Square{0.3};
\node at (2, 3) {$\partial_3$};
\draw(1,4)\Square{0.1};
\draw(4, 1)\Square{0.2};
\node at (4, 1) {$\partial_2$};

\draw [fill] (4.3,0.7) circle [radius=1.5pt];

\draw [fill] (3.7,1.3) circle [radius=1.5pt];

\draw [fill] (4.5,0.5) circle [radius=1.5pt];
\draw [fill] (.5,4.5) circle [radius=1.5pt];
\draw [fill] (2.5,2.5) circle [radius=1.5pt];
\draw [fill] (1.5,3.5) circle [radius=1.5pt];
\draw [fill] (3.5,1.5) circle [radius=1.5pt];
\draw [fill] (.2,4.8) circle [radius=1.5pt];
\draw [fill] (4.8,0.2) circle [radius=1.5pt];
\node [below] at (4.5, .5) {$m_1$};

\node [below] at (2.3, 2.5) {$m_0$};

\node [below] at (.5, 4.5) {$m_1$};
\node [below] at (2.6,1.6) {$\partial_2(\b)$};
\node [above] at (3.4,2.3) {$\partial_2(\t)$};f
\draw [fill] (3.4, 2.4) circle [radius=1.5pt];
\draw [fill] (2.6, 1.6) circle [radius=1.5pt];
\draw [fill] (0,0) circle [radius=1.5pt];
\draw [fill] (0,5) circle [radius=1.5pt];
\draw [fill] (5, 0) circle [radius=1.5pt];
\draw [fill] (5, 5) circle [radius=1.5pt];

\node at (1, 1) {$\partial(-)$};
\node at (4, 4) {$\partial(+)$};

\end{tikzpicture}
\end{center}
\caption{\label{fig:case2} A shuffle of $\set{\partial_i:i\in I}$.   The points labelled $m_0, m_1, \ldots$ represent one point boundary maps. }
\end{figure}
\begin{lemma}\label{lem:shuffles}
Let $\partial$ be a shuffle of $\set{\partial_i:i\in I}$.   Either there is $i_0, i_1\in I$ such that $\partial$ is a shuffle of  $\set{\partial_{i_0}, \partial_{i_1}}$ (note that at least one of these two must be a one point boundary map), or there is $I_0\subseteq I$  such that each proper boundary map $\partial_i$  (for $i\in I_0$) has height strictly less than the height of $\partial$, \/ $\partial$ is a shuffle of $\set{\partial_i:i\in I_0}$, and $| I_0|\leq |\phi|$.
\end{lemma}

\begin{proof}
If there is $i_0\in I$ such that $\partial_{i_0}$ is proper and has the same height as $\partial$ (so $\partial_{i_0}(\b)\in\partial_{i_0}(-)=\partial(-)$ and $\partial_{i_0}(\t)\in\partial_{i_0}(+)=\partial(+)$) then every future defect $\F\psi$ of $\partial(-)$ is passed up to $\partial_{i_0}(\b)$ (since $\F\psi\in\partial_{i_0}(\b)$) and every past defect of $\partial(+)$ is passed down to $\partial_{i_0}(\t)$.  Hence, $\partial$ is a shuffle of $\set{\partial_{i_0},  \partial_{i_1}}$,  for an arbitrary $i_1\in I$ such that $\partial_{i_1}$ is a one point boundary map.  

If there is no such $i_0\in I$ then each proper $\partial_i$ (for $i\in I$) has height strictly less than the height of $\partial$.  Define a subset $I_0\subseteq I$ so that for each future defect of $\partial(-)$  there is $i\in I_0$ such that the defect is passed up to $\partial_i(\b)$ and similarly for each past defect of $\partial(+)$ the defect is  passed down to $\partial_i(\t)$ for some $i\in I_0$, also include at least one $i\in I_0$ where $\partial_i$ is a one point boundary map. Since the total number of subformulas of $\phi$ (hence the total number of defects) is less than $|\phi|$ we can select $I_0$ resolving all defects of $\partial(-), \partial(+)$ and such that $| I_0|\leq |\phi|$.    By choice of $I_0$ it follows that $\partial$ is a shuffle of $\set{\partial_i:i\in I_0}$.

\end{proof}

\begin{definition}\label{def:undecomp} Let $h$ be a proper rectangle model such that $\partial^h=\partial^h\oplus_W\partial^h$, let $\partial^h(W)=(c_0, m_0, \ldots, c_k)$, for some $k\geq 0$.  For each $i<k$ let $\w_i=\x(h, W, i)$ be the point on the Western edge of $\dom(h)$ witnessing $m_i$ and let $\e_i=\x(h, E, i)$ be the corresponding point on the Eastern edge (see definition~\ref{def:trace2}, also see  figure~\ref{fig:vdm}).  We say that $h$ is a \emph{vertically displaced} rectangle model if for each $i<k$ it is not the case that $\e_i$ is due East from $\w_i$.  Horizontally displaced rectangle models are defined similarly.
\end{definition}

\begin{lemma}\label{lem:joins}
\begin{enumerate}
\item Let $h, h'$ be rectangle models such that $\partial^{h'}$ fits to the North of $\partial^h$.  There is a rectangle model $g$ such that $\partial^{h}\oplus_N\partial^{h'}=\partial^g$.  
\item 
If  $\partial^h\stackrel{W}\rightarrow\partial^*$ where $h$ is a vertically displaced rectangle model,  there is a rectangle model $g$ such that $\partial^*=\partial^g$.
\item  For each $i\in I$ suppose either $\partial_i=\partial^{h_i}$ for some proper, closed rectangle model $h_i$ or $\partial_i=\partial_{m_i}$, for some $m_i\in \MCSx$, and for at least one $i\in I$ we have $\partial_i=\partial_{m_i}$.  If $\partial$ is a  shuffle of $\set{\partial_i:i\in I}$ then there is a  rectangle model $h$ such that $\partial=\partial^h$.
\end{enumerate}
\end{lemma}

\begin{proof}
Suppose $\partial^{h'}$ fits to the North of $\partial^h$.  By lemma~\ref{lem:order} we may assume that $\dom(h), \dom(h')$ are two adjacent rectangles (the latter North of the former) containing a common boundary edge on which they agree. So $h\cup h'$ is a well-defined rectangle model, the required rectangle model for the first part of the lemma. 

Now suppose  $\partial^h\stackrel{W}\rightarrow\partial^*$ where $h$ is a vertically displaced rectangle model.  Let $\partial^h(W)=(c_0, m_0, \ldots, c_k)$ and let $(\w_0, \ldots, \w_{k-1}), \;(\e_0, \ldots, \e_{k-1})$ be the points witnessing $m_i$ on the Western and Eastern edges of $\dom(h)$, as in definition~\ref{def:undecomp}, for $i<k$.   By coherence of $h$ the distance from $\w_i$ up to $\l$ is not more than the distance from $\e_i$ up to $\t$ and by vertical displacement, the former distance is strictly less than the latter, for each $i<k$.   See figure~\ref{fig:limit}.

Let $\lambda<1$ be such that for all $i<k$ the distance of $\w_i$ to $\l$ is less than $\lambda$ times the distance of $\e_i$ to $\t$.  By lemma~\ref{lem:order} there is a rectangle model $h'$ order equivalent to $h$ (with corners $\t', \l', \r', \b'$) where the horizontal dimension of $\dom(h')$ is half that of $\dom(h)$, where the vertical dimensions are the same, where the restriction of $h'$ to its Eastern edge is identical to the restriction of $h$ to its Western edge and where $\lambda$ remains an upper bound on the ratio of the distance from $\w'_i$ to $\l'$ to the distance of $\e'_i$ to $\r'$, for all $i<k$.  Iterating this construction, we obtain a sequence of rectangle maps (as shown in figure~\ref{fig:limit}) each one of half the horizontal dimension to the one before, with the same vertical dimensions, each with boundary map $\partial^h$, where the  restriction of one rectangle model to its West boundary is identical to the restriction of the next rectangle model to its  East boundary.  Now define $g$ as the union of all these rectangles together with a map  that sends its bottom and left corners to $\partial^*(\b), \partial^*(\l)$ respectively  (if defined)  and if  $\partial^*(W)$ is defined, $g$ is defined densely along the Western edge so that the trace of the Western edge (under $g$) is $\partial^*(W)$.  The conditions for  $\partial^h\stackrel{W}\rightarrow\partial^g$ ensure that $g$ is a rectangle model, clearly its boundary map is $\partial^*$. 

For the final part of the lemma, we build a rectangle model $h$ that looks like figure~\ref{fig:case2}, reminiscent of the Cantor set \cite{Can83}, but whereas Cantor sets are obtained by repeated deletion of  the middle open third of a closed segment, the sets used here are obtained by repeatedly inserting a closed rectangle onto the central third of an open segment.  In more detail, use $\partial$ to define $h$ densely along the edge of the rectangle $[0,1]\times[0,1]$, when the edge is included in the domain of $\partial$.  Initially, use $\partial(-)$ densely on the interior of the square below the diagonal $x+y=1$ and use $\partial(+)$ densely on the interior of the square above this diagonal.   We have still to define the value of $h$ at points on the diagonal  $x+y=1$ and we will modify $h$ at points close to the diagonal, but the boundary and the lower/upper clusters will not be changed, hence we already know that $\partial^h=\partial$.    
 
 Pick any one point boundary map $\partial_{m_i}$ (some $i\in I$) and let $h(\x)=m_i$ along the diagonal (but not at the endpoints: if these belong to the domain then they are given already by $\partial(\l), \partial(\r)$).  
   We consider the open segment along the diagonal as a single gap.  At a later stage there will be a finite number of gaps, all open segments.  Pick any gap and  pick $i\in I$.  If $\partial_i=\partial^{h_i}$ we place a copy of the rectangle model $h_i$ in the central third of the chosen gap overwriting the labels of the points covered by the copy of $h_i$ and replacing the gap by two gaps each of one third the original size.  If $\partial_i=\partial_{m_i}$ we change the label of the midpoint of the gap to $m_i$ and replace that gap by two gaps each one half of the original size.  Keep repeating this process making sure that in every gap, for each $i\in I$,  a copy of a rectangle model described by  $\partial_i$ (i.e. either $h_i$ or a single point $m_i$)  gets placed eventually.  Since points get overwritten at most once, the limit of this process is a well defined structure $h$.  Now check that $h$ is indeed a rectangle model, we already saw that $\partial^h=\partial$.  
\end{proof}

\begin{lemma}\label{lem:displaced}
Let $h$ be a rectangle model, $\partial^h=\partial^h\oplus_W\partial^h$, let $\partial(W)=(c_0, m_0, \ldots, c_k)$.  The following are equivalent.
\begin{enumerate}
\item  $h$ is vertically displaced,\label{cond:1}
\item \label{cond:2}
for each $i<k$ there are $j\leq i,\; j'\geq i+1$ and  $\x_i$ such that $\dom(h)\cap N(\x_i)), \; \dom(h)\cap S(\x_i))$ are proper rectangles and either
\begin{itemize}
\item
$\partial(h, S(\x_i))(W)=(c_0, \ldots, c_j)$  and $\partial(h, S(\x_i))(E)=(c_0, \ldots, c_{j'})$ or
\item
$\partial(h, N(\x_i))(E)=(c_{j'}, \ldots, c_k)$  and $\partial(h, N(\x_i))(W)=(c_{j}, \ldots, c_k)$
\end{itemize}
\item either\label{cond:3}
\begin{enumerate}
\item\label{cond:3a}  for each $i<k$ there are  $j\leq i,\; j'\geq i+1$ and  $\x_i$ such that $\partial(h, S(\x_i)),$  $\partial(h, N(\x_i))$ are proper boundary maps of height strictly less than the height of $h$ and either
\begin{itemize}
\item
$\partial(h, S(\x_i))(W)=(c_0, \ldots, c_j)$ and $\partial(h, S(\x_i))(E)=(c_0, \ldots, c_{j'})$ or
\item
$\partial(h, N(\x_i))(E)=(c_{j'}, \ldots, c_k)$ and $\partial(h, N(\x_i))(W)=(c_{j}, \ldots, c_k)$
\end{itemize}
 or 
\item\label{cond:3b} there are $\y_0, \y_1$ such that $\partial(h, S(\y_0)),\; \partial(h, N(\y_1))$ are proper boundary maps of height strictly less than the height of $h$ and  $\partial(h, N(\y_0)\cap S(\y_1))$ is also proper and $\partial(h, N(\y_0)\cap S(\y_1))(W)=(c_0)$,  $\partial(h, N(\y_0)\cap S(\y_1))(E)=(c_{k})$,
\end{enumerate}
\end{enumerate}
\end{lemma}
\begin{proof} For each $\x\in \overline{\dom(h)}\subseteq\reals^2$ let $y(\x)$ be the $y$ coordinate of $\x$.
For (\ref{cond:1}$\Rightarrow$\ref{cond:3}), 
suppose $h$ is a vertically displaced rectangle model, let $(\w_0, \ldots, \w_{k-1}), \;(\e_0, \ldots, \e_{k-1})$ be the points on the Western and Eastern edges of $\dom(h)$ witnessing $(m_0, \ldots, m_{k-1})$, as illustrated in figure~\ref{fig:vdm}.  First suppose $y(\w_0)\leq y(\e_{k-1})$ (see the left side of figure~\ref{fig:vdm}).    Then $y(\w_0), y(\e_i) \leq y(\w_i), y(\e_{k-1})$ so there is $y$ with $max(y(\w_0), y(\e_i))\leq y_i\leq  min(y(\e_{k-1}), y(\w_i))$.
 Pick any point $\x_i$ with $y(\x_i)=y_i$.  Since $y_i\leq y(\e_{k-1})$, the height of $\partial(h, S(\x_i))$ is strictly less than the height of $h$, and since $y_i\geq y(\w_0)$ the height of $\partial(h, N(\x_i))$ is also strictly less than the height of $h$. By vertical displacement, $y(\e_i)<y(\w_i)$.   If $y_i>y(\e_i)$ (as illustrated) then the first bullet of \eqref{cond:3a} holds else $y_i<y(\w_i)$ and the second bullet holds, either way \eqref{cond:3a} holds.  Otherwise $y(\w_0)>y(\e_{k-1})$ (see the right side of figure~\ref{fig:vdm}),  we let $\y_0=\e_{k-1},\; \y_1=\w_0$ and we get \eqref{cond:3b}.

\begin{figure}  
\begin{center}
\newcommand\Square[1]{+(-#1,-#1) rectangle +(#1,#1)}
\begin{tikzpicture}[scale=.8]
\draw(0,2) \Square{2};

\draw [fill] (2,1) circle [radius=2.0pt];
\draw [fill] (2,2) circle [radius=2.0pt];
\draw [fill] (2,3) circle [radius=2.0pt];
\draw [fill] (-2,1.5) circle [radius=2.0pt];
\draw [fill] (-2,3.5) circle [radius=2.0pt];
\draw [fill] (-2,2.5) circle [radius=2.0pt];
\draw [fill] (-1,2.25) circle [radius=2.0pt];

\draw[dashed] (-2, 2.25) -- (2, 2.25);

\node[left] at (-2, 1.5){$\w_0$};
\node[left] at (-2, 2.5){$\w_i$};
\node[left] at (-2, 3.5){$\w_{k-1}$};

\node[right] at (2,1){$\e_0$};
\node[right] at (2,2){$\e_i$};
\node[right] at (2, 3){$\e_{k-1}$};

\node[above]  at (0,4) {$c_k$};
\node[below]  at (0,0) {$c_0$};
\node[below] at (-1, 2.25){$\x_i$};


\draw[dashed] (-2, 1.5) .. controls (-1,.5) and (1,1.5) ..(2, 1);

\draw[dashed] (-2, 3.5) .. controls (-1,2.5) and (1,3.5) ..(2, 3);

\node at (0, 2.6) {$N(\x_i)$};

\node at (0,1.5){$S(\x_i)$};

\node at (0, .5){$c_0$};
\node at (0, 3.5){$c_k$};

\end{tikzpicture}
\hspace{.5in}
\begin{tikzpicture}[scale=.8]
\draw(0,2) \Square{2};
\node[below]  at (0,0) {$c_0$};
\node[above]  at (0,4) {$c_k$};

\draw [fill] (2,.5) circle [radius=2.0pt];
\draw [fill] (2,1) circle [radius=2.0pt];
\draw [fill] (2,1.5) circle [radius=2.0pt];
\draw [fill] (-2,2.5) circle [radius=2.0pt];
\draw [fill] (-2,3) circle [radius=2.0pt];
\draw [fill] (-2,3.5) circle [radius=2.0pt];

\node[left] at (-2, 3.5){$\w_{k-1}$};
\node[left] at (-2, 3){$\w_i$};
\node[left] at (-2, 2.5){$\w_0$};
\node[right] at (2, .5){$\e_0$};
\node[right] at (2, 1){$\e_i$};
\node[right] at (2, 1.5){$\e_{k-1}$};

\node[right] at(-2, 2.5){$\y_1$};
\node[left] at (2, 1.5){$\y_0$};

\draw[dashed] (-2, 1.5) -- (2, 1.5);
\draw[dashed](-2, 2.5) -- (2, 2.5);

\node at (0, 3) {$N(\y_1)$};
\node at (0, .7) {$S(\y_0)$};
\node at (0, 2){$N(\y_0)\cap S(\y_1)$};


\draw [decorate,decoration={brace,amplitude=5pt},xshift=-4pt,yshift=0pt]
(-2,0) -- (-2,2.3)  node [black,midway,xshift=-0.4cm] {$c_0$};

\draw [decorate,decoration={brace,amplitude=5pt},xshift=--4pt,yshift=0pt]
(2,3.9) -- (2,1.7)  node [black,midway,xshift=--0.4cm] {$c_k$};

\end{tikzpicture}
\caption{Vertically displaced models, where $y(\w_0)\leq y(\e_{k-1})$ or $y(\w_0)>y(\e_{k-1})$, respectively.\label{fig:vdm}}
\end{center}
\end{figure}
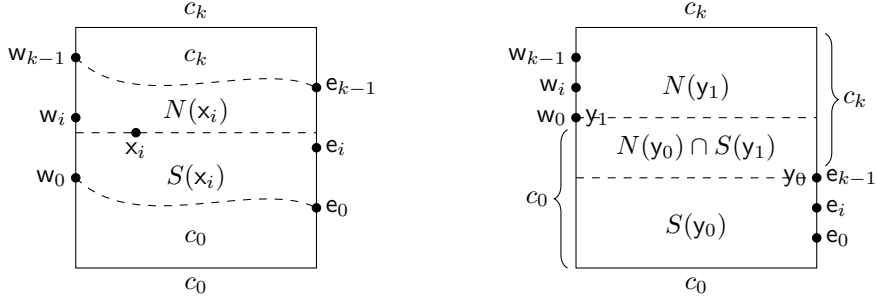

 The implication (\ref{cond:3a}) $\Rightarrow$ (\ref{cond:2}) is trivial.  If \eqref{cond:3b} holds then let $j=i,\; j'=k,\;\x_i=\w_i$ to see that the first bullet in \eqref{cond:2} holds.
 
If  condition~(\ref{cond:2})  holds, say the first bullet holds, then $j\leq i\rightarrow y(\w_i)\geq y(\x_i)$ and $j'\geq i+1\rightarrow y(\e_i)<y(\x_i)$, similarly if the second bullet holds we get $y(\w_i)>y(\e_i)$, hence $h$ is vertically displaced.

\end{proof}
We make use of condition~\eqref{cond:2} from the previous lemma in the next definition.  Condition~\ref{cond:3} will be used in section~\ref{sec:complexity}.
\begin{definition} Let $\partial=(\partial\oplus_W\partial)\in X$, say $\partial(W)=(c_0, m_0, \ldots, c_k)$.   We write $Split(\partial, horiz, i, \partial_i^-, \partial_i^+)$ if $\partial=\partial_i^-\oplus_N\partial_i^+$, there are $j\leq i,\; j'\geq i+1$ such that either $\partial_i^-(W)=(c_0,\ldots, c_j),\; \partial_i^-(E)=(c_0,\ldots, c_{j'})$, or $\partial_i^+(E)=(c_{j'}, \ldots, c_k),\; \partial_i^+(W)=(c_j, \ldots, c_k)$, with a similar definition of $Split(\partial, vert, i, \partial_i^-, \partial_i^+)$.

 If $X$ is a set of boundary maps, $\partial\in X,\; \partial\stackrel{W}\rightarrow\partial^*$,  and for each $i<k$ there are $\partial_i^-, \partial_i^+\in X$ such that $Split(\partial, horiz, i,\partial_i^-, \partial_i^+)$ holds, we say that $\partial^*$ is a Western limit over $X$.
We say that $X$ is \emph{closed under  Western limits} if whenever $\partial^*$ is a Western limit  over $X$ we have $\partial^*\in X$.  We say that $X$ is \emph{closed under  limits} if  it is closed under  Western, Eastern, Southern and Northern limits, which are defined similarly.  For any set of proper boundary maps $X$ we say that a proper  boundary map $\partial$ is a shuffle over $X$ if there is $Y\subseteq X$ and a non-empty set $M$ of one point boundary maps such that $\partial$ is a shuffle of $Y\cup M$.

\end{definition}
Our plan now is to define a  recursive set $B$ of  proper boundary maps corresponding exactly to the boundary maps of  rectangle models.  We can then use $B$ to determine the satisfiability of $\phi$. 
We start with  our basic building blocks.  A proper boundary map $\partial$ belongs to $B_0$ iff 

\begin{enumerate}
\item $\partial(-)=\partial(+)$, call this cluster $c$,
\item if $\partial(S)$ is defined then $\partial(S)\leq c$  and $\partial(W)\leq c$ if defined, 
\item  every  future defect of $\partial(S)$  is passed up to $\partial(\r)$ or $c$ and every future defect of $\partial(W)$ is passed up to $\partial(\l)$ or $c$ 
\item dual properties obtained by swapping the pairs $(\F,\P),(-, +), (S, E), (N, W)$ also hold.
\end{enumerate}

Now we define the   set of proper boundary maps $B$, by recursion.  $B$ is the smallest set of proper boundary maps containing $B_0$ and closed under joins, shuffles and limits.  $B$ can be computed using algorithm~\ref{alg}.
\begin{algorithm}
\caption{\label{alg} Algorithm to compute $B$}
\begin{algorithmic}
\State $B=B_0$
\While {new elements can be found}
\State Add any joins of elements of $B$ to $B$
\State  Add any  limits over $B$  (in any of the four directions) to $B$
\State Add any shuffles over  $B$ to $B$
\EndWhile
\end{algorithmic}
\end{algorithm}

\begin{definition}\label{def:simple}
Let $h$ be a rectangle model, let $W$ be the open line segment on the Western side of $\dom(h)$.  We say that $h$ has \emph{a simple Western edge} if  
for all $\x\in W$ there is a neighbourhood $\eta$ of $\x$ and $h[\eta\cap Int(\dom(h))]\subseteq c$, where $c$ is the lower cluster of $h$.  Simple Eastern, Northern and Southern edges are defined similarly.
\end{definition}

\begin{lemma}\label{lem:limit3}
Let $h$ be a rectangle model, let $\x\in \overline{\dom(h)}$ be on the Western edge of $\dom(h)$ and suppose $h$ has a simple Western edge.
Let $\x_0>_1\x_1>_1\ldots \in \dom(h)$ be a sequence of points converging from the East to $\x$.  If $\partial(h, W(\x_i)\cap E(\x_j))$ is constant, for $i<j<\omega$ then
$\partial(h, W(\x_0)\cap E(\x_1))\stackrel{W}\rightarrow \partial(h, W(\x_0)\cap E(\x))$ and has a vertically displaced rectangle model.  Moreover, $\partial(h, W(\x_0)\cap E(\x))$ is a Western limit  over $\set{\partial(h, R):  i<\omega,\;R\subseteq \dom(h)\cap W(\x_0)\cap E(\x_i)}$.
\end{lemma}

\begin{proof}
  Since $\partial(h, W(\x_i)\cap E(\x_j))$ is constant, by considering any $i<j<k$ we see that $\partial(h, W(\x_0)\cap E(\x_1))$ fits to the West of itself, and $\partial(h, W(\x_0)\cap E(\x))$ agrees with $\partial(h, W(\x_0)\cap E(\x_1))$ except perhaps on $W, \b, \l$.   Since $h$ has a simple Western edge ($W$) there must be $t<\omega$ such that $h\restr{W(\x_0)\cap E(\x_t)}$ is vertically displaced, and $\partial(h, W(\x_0)\cap E(\x_t))=\partial(h, W(\x_0)\cap E(\x_1))$.
Let $\partial(h, W(\x_0)\cap E(\x_t))(W)=(c_0, \ldots, c_k)$ say, 
 by lemma~\ref{lem:displaced}\eqref{cond:2} for each $i<k$ there is $\z_i$ such that $Split(\partial(h, W(\x_0)\cap E(\x_t)), horiz, i, \partial(h, W(\x_0)\cap E(\x_t)\cap S(\z_i)), \partial(h, W(\x_0)\cap E(\x_t)\cap N(\z_i)))$ holds, hence $\partial(h, W(\x_0)\cap E(\x))$ is a Western limit  over $\set{\partial(h, W(\x_0)\cap E(\x_t)\cap N(\z_i)),\partial(h, W(\x_0)\cap E(\x_t)\cap S(\z_i)):i<k}$.

\end{proof}

\section{Decidability}
\begin{lemma}\label{lem:main}
Let $\partial$ be a boundary function. 
 $\partial\in B$ iff there is rectangle model $h$ such that $\partial^h=\partial.$
\end{lemma}

\begin{proof}
The left to right implication is proved by induction on the number of iterations of the while loop in algorithm~\ref{alg}.
 For the base case, if $\partial\in B_0$ let $h$ be a proper rectangle model  which maps densely with $\partial(-)$ \/ ($=\partial(+)$) over the interior and maps densely in accordance with $\partial$ along the boundary, where it is defined.  The conditions for membership of $B_0$ ensure that $h$ is coherent and all defects are passed North, South, East or West as appropriate, hence $h$ is a rectangle model and $\partial=\partial^h$.    
 The  induction step holds by lemmas~ \ref{lem:joins} and \ref{lem:displaced}(\ref{cond:2}$\Rightarrow$\ref{cond:1}).
 
\medskip 
 
The right to left implication is proved by induction over the height of $h$.  Let $h$ be any rectangle model  of height zero.     $\partial^h$ satisfies the first condition for $B_0$, since $\partial^h(-)=\partial^h(+)$.    The other three conditions also hold, since $h$ is a rectangle model, so every defect of $h$ must be passed in a suitable direction.  Thus $\partial^h\in B_0\subseteq B$.

 Now let the height of $h$ be positive and assume  the lemma holds for rectangle models of smaller height.   
\begin{definition}\label{def:good}
A  rectangle $R$ is said to be \emph{good} and we may write $G(R)$,  if  for all proper rectangles $R'\subseteq R$ we have $\partial(h, R')\in B$.
\end{definition}    Our induction hypothesis tells us that if the height of $\partial(h, R)$ is strictly less than the height of $h$ then $R$ is good.  We aim to prove the converse implication in the lemma by showing that $\dom(h)$ is good.

\begin{claim}\label{cl:join}
If $X$ is a rectangle, $\x\in Int(X)$ then $(G(X\cap S(\x))\wedge G(X\cap N(\x)))\rightarrow G(X)$ and 
$(G(X\cap W(\x))\wedge G(X\cap E(\x)))\rightarrow G(X)$.
\end{claim}
This claim holds by lemma~\ref{lem:split},  since $B$ is closed under joins.
\\

\begin{figure}
\begin{center}
\begin{tikzpicture}[xscale=.9,yscale=.6]
\draw (0,0) rectangle (10,6);
\draw (0,1) -- (10,1);
\draw (0,5) -- (10,5);

\draw (8,1)--(8,5);
\draw(4,1)--(4,5);
\draw(2,1)--(2,5);
\draw(1,1)--(1,5);
\draw(.5,1)--(.5,5);
\draw(.25,1)--(.25,5);
\draw(.125, 1)--(.125, 5);

\node [left] at (0,1) {$\x$};

\node[below] at (8,1) {$\x_0$};
\node[below] at (4, 1){$\x_1$};
\node[below] at (2, 1){$\x_2$};
\node[below] at (1,1){$\x_3$};
\node[below] at (.5,1){\phantom{a}$\ldots$};

\node at (5,1.5) {$S_0$};
\node at (5,2.833) {$S_1$};
\node at (5,4.17) {$S_2$};

\draw [fill] (0,2.33) circle [radius=2pt];
\draw [fill] (0,3.67) circle [radius=2pt];
\node[left] at (0,2.33) {$m_0$};
\node[left] at (0, 3.67) {$m_1$};
\node[left] at (0, 1.67) {$c_0$};
\node[left] at (0, 3){$c_1$};
\node[left] at (0, 4.33) {$c_2$};

\draw [dashed] (0,2.33) -- (10, 2.33);
\draw [dashed] (0,3.67) -- (10, 3.67);
\end{tikzpicture}
\end{center}

\caption{\label{fig:limit2} For $k<3$,\/ $S_l$ is a  limit of  rectangles $S_l\cap E(\x_i)\; (i<\omega)$}
\end{figure}
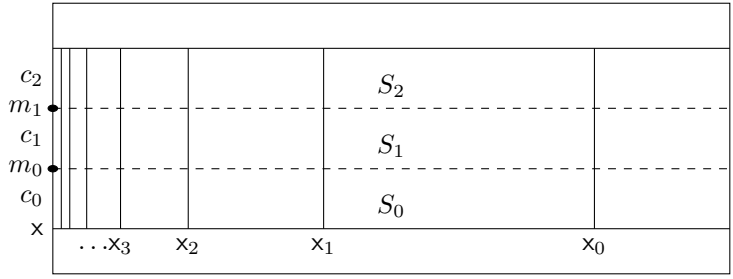

Let $\x\in \overline{\dom(h)}$ and let $\x_0>_1\x_1>_1\ldots\x$ be a sequence converging on $\x$ from the East, as shown in figure~\ref{fig:limit2}.   Let $R$ be a rectangle.
\begin{claim}\label{cl:limit}
 $(\forall i>0\; G(R\cap E(\x_i))) \rightarrow G(R\cap E(\x))$.
 \end{claim}
 To prove this claim, consider an arbitrary proper rectangle $S$ contained in $R\cap E(\x)$.  If $S\subseteq E(\x_i)$ (some $i<\omega$) then its boundary map belongs to $B$.  Assume not, so the Western boundary edge of $S$ is contained in the Western boundary edge of $R\cap E(\x)$.
We may draw (finite) $p$  horizontal lines to divide $S$ into $p+1$ subrectangles $S_0, S_1, \ldots, S_p$ (in figure~\ref{fig:limit2}, $p=2$) where each $S_l$ has a simple Western edge ($l\leq p$).  To show that $\partial(S_l)\in B$  (any $l\leq p$)
 apply Ramsay's theorem to obtain an infinite subsequence $\x'_0>_1\x'_1>_1\ldots$ converging to $\x$ where the boundary map $\partial(h,S_l\cap W(\x'_i)\cap E(\x'_j)$ is constant over all $i<j<\omega$. By lemma~\ref{lem:limit3}, $\partial(h, S_l\cap W(\x'_0)\cap E(\x'_1)) \stackrel{W}{\rightarrow} \partial(h, S_l\cap W(\x_0'))$ and
   $\partial(h, S_l\cap  W(\x'_0))$ is a Western limit over $B$, so it belongs to $B$.  By claim~\ref{cl:join}, $\partial(h, S_l)=\partial(h, S_l\cap W(\x'_0))\oplus_E\partial(h, S_l\cap  E(\x'_0))\in B$, for each $l\leq p$, and therefore $\partial(h, S)=\partial(h, S_0)\oplus_N\ldots\oplus_N\partial(h, S_p)\in B$.  Hence  $R\cap E(\x)$ is good, as claimed.\\

\begin{claim}\label{cl:close}
For any  rectangle $R$, if every proper closed subrectangle of $R$ contained in $Int(R)$ is good then $ R$ is also good.
\end{claim}
If $R$ is not proper then vacuously it is good, so assume it is proper.
For any $a_1<a_2$ and  $b_1<b_2$ such that each line $x=a_1,\; x=a_2,\; y=b_1,\; y=b_2$ passes through the interior of $R$, the closed rectangle bound by these four lines is good, by assumption.  By fixing the other three parameters but varying $a_1$ so that the line $x=a_1$ approaches the Western edge of $R$, we can apply claim~\ref{cl:limit} to deduce that the rectangle $R\cap\set{(x, y):x\leq a_2, \; b_1\leq y\leq b_2}$ is good, provided the lines $x=a_2,\; y=b_1,\; b=b_2$ pass through the interior of $R$.  Now apply claim~\ref{cl:limit} three more times to move the other three bounding lines towards the other three edges and we see that $R$ is also good, as claimed.
\\

Let $I$ be the interior of $h^{-1}(\partial(-))$ and let $\Gamma$ be the closure of the upper boundary of $h^{-1}(\partial(-))$, similarly let $J$ be the interior of $h^{-1}(\partial(+))$ and let $\Delta$ be the closure of the lower boundary, see figure~\ref{fig:gamma}.  Since $\Gamma, \Delta$ are closed, so is $\Gamma\cap\Delta$.  Recall from definition~\ref{def:simple} what a simple Western edge is.
\begin{claim}\label{cl:simple} Suppose $\dom(h)\cap N(\x)\cap E(\x)$ is a proper rectangle. If $\x\in J\cup\Delta$ then $h\restr{N(\x)\cap E(\x)}$ has a simple Western edge.  If $\x<_2\y\in J$ but $\x\not\in (J\cup\Delta)$ then $h\restr{N(\x)\cap E(\x)}$ does not have a simple Western edge.
\end{claim}
There are three dual claims involving simple Southern edges and (for $\x\in I\cup\Gamma$) simple Eastern and Northern edges.
  To prove the claim, if $\x\in J$ or $\x\in\Delta$  then $\x\prec\y\rightarrow h(\y)\in\partial(+)$, hence every $\z>_2\x$ has a neighbourhood $\eta$ such that $h[\eta\cap Int(N(\x)\cap E(\x))]\subseteq \partial(+)$, so $h\restr{N(\x)\cap E(\x)}$ has a simple Western edge.  If $\x\not\in (J\cup\Delta)$ but $\x<_2\y\in J$ then there is $\y'$  where $y'\not\in \overline J$,\/ $\x<_2\y'<_2\y$ and neighbourhoods $\eta'$ of $\y'$,\/ $\eta$ of $\y$ such that  $h[\eta'\cap N(\x)\cap E(\x)]\cap \partial(+)=\emptyset$ but $h[\eta\cap N(\x)\cap E(\x)]\subseteq \partial(+)$, hence the rectangle map does not have a simple Western edge.

\begin{claim}\label{cl:int}
If $\Gamma\cap\Delta\cap Int(\dom(h))=\emptyset$ then $\dom(h)$ is good.   
\end{claim}
For this claim, first  suppose that $\Gamma\cap\Delta=\emptyset$ (see the first part of figure~\ref{fig:epsilon}).  
By lemma~\ref{lem:segment}, $\Gamma$ and $\Delta$ are homeomorphic to closed line segments.  Since $\Gamma\cap\Delta=\emptyset$  by lemma~\ref{lem:apart} there is $\epsilon>0$ such that $(\x\in\Gamma\wedge\y\in\Delta)\rightarrow |\x-\y|>\epsilon$.   Cover $\dom(h)$ with finitely many rectangles, where the length of each diagonal is less than $\epsilon$.    The restriction of $h$ to such a rectangle cannot contain  points in both $I$ and $J$ so the height of $h$ restricted to any one of these rectangles is strictly less than the height of $h$.  By the induction hypothesis, all these rectangles are good and by claim~\ref{cl:join}, $\dom(h)$ is good. 

Now consider the case where  $\Gamma, \Delta$ do not meet within the interior of $\dom(h)$, though they might meet on the boundary of $\dom(h)$.  Every rectangle contained in the interior of $\dom(h)$ is good, by the previous case. Hence by claim~\ref{cl:close}, $\dom(h)$ is also good.
\\

\begin{figure}
\begin{tikzpicture}[xscale=.7,yscale=.4]

\draw (0,0) rectangle (8,8);

\draw  ((6,0) .. controls (5,3)  and (2,3.2)   .. (0,5);

\draw (0,7) .. controls (2,2) and (4,3) .. (8,3);

\node at (2,2) {$I=h^{-1}(\partial(-))$};

\node at (6,6) {$J=h^{-1}(\partial(+))$};

\node at (5.1,1) {$\Gamma$};

\node at (7.5, 3.4) {$\Delta$};

\node [left] at (0,8) {$\l$};
\node[right] at (8,0) {$\r$};

\draw[fill] (2.7,3.33) circle [radius=2pt];
\draw[fill] (2.6,3.39) circle [radius=2pt];
\draw[fill] (2.4,3.49) circle [radius=2pt];
\draw[fill] (2.99,3.2) circle [radius=2pt];
\node[above right] at (2.6, 3.4) {$\Gamma\cap\Delta$};

\end{tikzpicture}
\caption{\label{fig:gamma} Upper and lower clusters with boundaries $\Gamma, \Delta$ shown}
\end{figure}
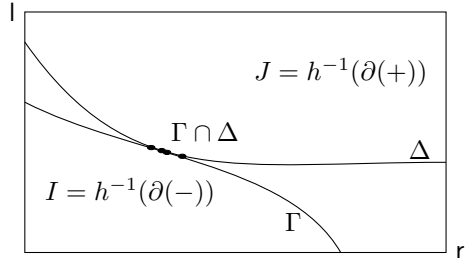

\medskip
Recall from the beginning the binary relation $\triangleleft$ over $\reals^2$.  Recall also, from lemma~\ref{lem:segment}, that $\Gamma$ is  homeomorphic to a closed line segment,  and observe that the restriction of $\triangleleft$ to $\Gamma$ is a linear order where  $\x\triangleleft\y$ if $\x$ is strictly to the left of (North West of, nearer to $\l$ than)  $\y$.  Let 
\[P=(\Gamma\cap\Delta)\cup\set{\l, \r},\]
so $(P, \triangleleft)$ is a linearly ordered, closed set.  

\begin{claim}\label{cl:limit2}
Let $\x_0\triangleright\x_1\triangleright\ldots \in P$ be a countable sequence converging to $\x\in P$.  If $[\x_i\wedge\x_j, \x_i\vee\x_j]$ is good (for all $0\leq j<i<\omega$) then $R=[\x\wedge\x_0, \x\vee\x_0]$ is also good.
\end{claim}
  The rectangle $W(\x_i)\cap E(\x_j)\cap R$ is good (for $i<j$) since it is the join of $[\x_i\wedge\x_j, \x_i\vee\x_j]$ (assumed good), $W(\x_i)\cap E(\x_j)\cap N(\x_j)$ and $W(\x_i)\cap E(\x_j)\cap S(\x_i)$, and the restriction of $h$ to either of the latter two rectangles has height strictly less than the height of $h$.  The sequence $\x_0\vee \x, \x_1\vee\x, \x_2\vee\x,  \ldots$ converges from the East to $\x$ and $W(\x_i\vee\x)=W(\x_i)$. By claim~\ref{cl:limit} $R=R\cap W(\x)$ is also good, as claimed.
\\

Define a binary relation $\approx$ over $P$ as the smallest equivalence relation such that
\begin{itemize}
\item  if $\y$ is an immediate $\triangleleft$ successor of $\x$  in $P$ then $\x\approx\y$,
\item if $\x\leq \y$ or $\y\leq\x$  (so $[\x\wedge\y,\x\vee\y]$ is not proper) then $\x\approx\y$,
\item  if $S$ is contained in  a $\approx$-equivalence class and if $\z\in\overline S$  then  $\z$ is in the same $\approx$-equivalence class (note that $P$ is closed so we know that $\z\in P$).
\end{itemize}
\begin{claim}\label{cl:E}  Let $S\subseteq P$ be a $\approx$ equivalence class, let $\x_0, \x_1\in P$ be the rightmost left bound  of $S$ (w.r.t. $\triangleleft$) and leftmost right bound respectively.  The rectangle $R(S)=[\x_0\wedge\x_1, \x_0\vee\x_1]$  is good.
\end{claim}
Note first that $\x_0, \x_1\in S$ since they are accumulation points, and $P$ is closed.  
To prove the claim, we know from claim~\ref{cl:int} that if $\y$ is an immediate $\triangleleft$ successor of $\x$ then $[\x\wedge\y, \x\vee\y]$ is good. Trivially, if $\x\leq\y$ then $[\x, \y]$ is not proper, hence it is good.    We show that the set $\varepsilon$ of pairs $(\x, \y)\in S\times S$ for which $[\x\wedge\y,\x\vee\y]$ is good is an equivalence relation.  It is clearly reflexive and symmetric.  If $\x\triangleleft\y\triangleleft\z\in S$ and $[\x\wedge\y,\x\vee\y], [\y\wedge\z, \y\vee\z]$ are good then $[\x\wedge\z,\x\vee\z]$ is the union of $[\x\wedge\y,\x\vee\y], [\y\wedge\z, \y\vee\z], [\x\wedge\z, \y], [\y, \x\vee\z]$ and since the last two rectangles have height less than the height of $h$, they are good,   hence by claim~\ref{cl:join} $[\x\wedge\z, \x\vee\z]$ is also good, so $\varepsilon$ is transitive.
To complete the proof of the claim observe that $\varepsilon$ is closed under limits, by claim~\ref{cl:limit2}.
\\

Next observe that $\triangleleft$ induces a strict dense ordering of $E=P/\approx$.     Recall that each $S\in E$ contains a leftmost point $\l(S)$ and a rightmost point $\r(S)$ and defines a good rectangle $R(S)=[\l(S)\wedge \r(S), \l(S)\vee \r(S)]$.  An \emph{$\omega$-sequence} is a map $\lambda:\omega\rightarrow P$ such that for all $i<\omega$,\/ $\lambda(i+1)$ is an immediate $\triangleleft$ successor of $\lambda(i)$.  Dually, an $\overline{\omega}$-sequence is a function $\lambda:\omega\rightarrow P$ such that $\lambda(i+1)$ is an immediate $\triangleright$ successor of $\lambda(i)$. 
\begin{claim}\label{cl:seq}
For any $S\in E$ and $s\in S$ either (i) $s=\l(S)$, (ii) $s$ is an immediate $\triangleleft$-successor of some $t\in S$, or (iii) for all $\epsilon>0$ there is an $\omega$-sequence or an $\overline{\omega}$-sequence $\lambda$ whose range is to the left of $s$ but within $\epsilon$ of it.
\end{claim}
To prove the claim, note that if there is $\epsilon>0$ and there is no $\omega$-sequence or $\overline\omega$-sequence whose range is to the left of $s$ and within $\epsilon$ of $s$ then for all $u, v\in P$  where $u\triangleleft v\triangleleft s$ are within $\epsilon$ of $s$, we have $u\approx v$ iff there is a finite chain of immediate successors $u=u_0\triangleleft u_1\triangleleft\ldots\triangleleft u_l=v$.  Hence, 
if $s\in S$ is not a $\triangleleft$-successor and there is $\epsilon>0$ such that there is no $\omega$-sequence or $\overline{\omega}$-sequence whose range is to the left of $s$ within $\epsilon$ of it,  if $t\triangleleft s$ then $t\not\approx s$, hence $s=\l(S)$.\\

Let $\Gamma^+$ be the closed line obtained by extending $\Gamma$ by adding a straight line segment from the leftmost point of $\Gamma$ to $\l$ and adding another straight line segment from the rightmost point of $\Gamma$ to $\r$.   
\begin{claim}\label{cl:boxes}
$\Gamma^+$ is a line segment running from $\l$ to $\r$ and it is  the disjoint union of the closed line segments/points $\Gamma^+\cap R(S)$, as $S$ ranges over $E$.
\end{claim}
For this claim, use lemma~\ref{lem:segment} to prove that $\Gamma^+$ is homeomorphic to a closed line segment, and that $\Gamma^+\cap R(S)$ is homeomorphic to a closed line segment that runs from $\l(R(S))$ to $\r(R(S))$, for each $S\in E$.     We note  (by closure of $P$) that for every $\x\in\Gamma^+$ there are nearest neighbours in $P$ wrt $\triangleleft$, one to the left or equal to $\x$, the other to the right or equal to $\x$,  and that $\x$ belongs to the rectangle $R(S)$, where $S$ is the equivalence class of these two (not necessarily distinct) neighbours.  
 It follows  that  $\Gamma^+=\bigcup_{S\in E} (\Gamma^+\cap R(S))$ and the rectangles $R(S)$ are clearly disjoint from each other, since $\approx$ is transitive. This proves claim~\ref{cl:boxes}.

\begin{claim}\label{cl:singleton}  Either $E$ consists of a single $\approx$-equivalence class $S$ and $\overline{\dom(h)}=R(S)$, or $E$ contains uncountably many singleton $\approx$ equivalence classes.
\end{claim}
To prove this claim, note by the previous claim that the union of $\Gamma^+\cap R(S)$ as $S$ ranges over $E$ is homeomorphic to a closed line segment.  If there is more than one equivalence class then by lemma~\ref{lem:point},  uncountably many  segments must be singleton points, hence the equivalence classes of these points must be singletons.

\begin{claim}\label{cl:shuffle} Suppose $|E|>1$.   Let $\x_0=\r(S_0)$, where $S_0$ is the $\approx$-equivalence class of $\l$ and let $\x_1=\l(S_1)$, where $S_1$ is the $\approx$-equivalence class of $\r$.   See figure~\ref{fig:EW}.  Then 
\begin{align*}
\partial  & = \partial(h, W(\x_0))\\
&\oplus_E(\partial(h, [\x_0\wedge\r, \x_1])\oplus_N\partial(h, [\x_0\wedge\x_1,\x_0\vee\x_1])\oplus_N\partial(h, [\x_0, \l\vee\x_1]))\\
&\oplus_E\partial(h, E(\x_1))
\end{align*}
 and $\partial(h, [\x_0\wedge\x_1,\x_0\vee\x_1])$ is a shuffle of  $\set{\partial(h, R(S)): S\in E\setminus\set{S_0, S_1}}$.
\end{claim}
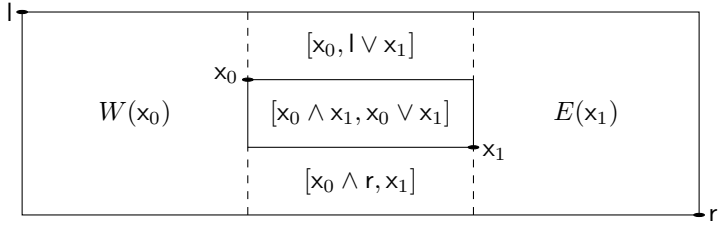
\begin{figure}
\begin{tikzpicture}[yscale=.3]
\draw (0,0) rectangle (9,9);
\draw (3,3) rectangle (6,6);
\draw[dashed] (3,0)--(3,3);

\draw[dashed] (6,0)--(6,3);
\draw[dashed] (3,9)--(3,6);
 \draw[dashed] (6,9)--(6,6);
 
 \node[left] at (0,9) {$\l$};
 \node[right] at (9,0){$\r$};
 \node [left] at (3, 6.2) {$\x_0$};
 \node[right] at (6, 2.8) {$\x_1$};
 
 \node at (1.5, 4.5) {$W(\x_0)$};
  \node at (7.5, 4.5) {$E(\x_1)$};
  
  \node at (4.5, 1.5) {$[\x_0\wedge\r, \x_1]$};
    \node at (4.5, 7.5) {$[\x_0, \l\vee\x_1]$};
 
 \node at (4.5, 4.5) {$[\x_0\wedge\x_1, \x_0\vee\x_1]$};
 
 \draw[fill] (0,9) circle [radius=2pt];
  \draw[fill] (9,0) circle [radius=2pt];
   \draw[fill] (3,6) circle [radius=2pt];
    \draw[fill] (6,3) circle [radius=2pt];
\end{tikzpicture}
\caption{\label{fig:EW}$[\l]_\approx=S_0,\; [\r]_\approx=S_1,\; \x_0=\r(S_0),\;\x_1=\l(S_1)$ and $\partial(h, [\x_0\wedge\x_1, \x_0\vee\x_1])$ is a shuffle of $\set{(\partial(h, R(S)): S\in E\setminus\set{S_0, S_1}}$.}

\end{figure}
The displayed equation follows from lemma~\ref{lem:split}, we check that $\partial(h, [\x_0\wedge\x_1, \x_0\vee\x_1])$ is a shuffle of boundary maps $\partial(h, R(S))$ for $S\in E\setminus\set{S_0, S_1}$.

First note that since $E$ contains uncountably many unordered singletons (claim \ref{cl:singleton}), there are uncountably many elements of $\Gamma\cap\Delta$ between $\x_0$ and $\x_1$.  The fact that $\Gamma\cap\Delta$  meets the interior of $[\x_0\wedge\x_1, \x_0\vee\x_1]$  ensures there is $\y\in \Gamma\cap\Delta$ with $\y\succ\x_0\wedge\x_1$, hence  $\x_0\wedge\x_1\in I$ and similarly $\x_0\vee\x_1\in J$.   Every point of $[\x_0\wedge\x_1, \x_0\vee\x_1]$ below $\Gamma$ is in $I$, every point above $\Delta$ is in $J$ and every point in $\Gamma^\uparrow\cap\Delta^\downarrow$ is in $R(S)$, for some $S\in E$.
To prove that $]\x_0\wedge\x_1, \x_0\vee\x_1]$ is a shuffle, we check conditions~\ref{sh:1}--\ref{sh:6} of definition~\ref{def:combs}, one at a time.   Condition~\ref{sh:1}  follows from the coherence of $h$.  For  condition~\ref{sh:2}, let $\F\psi$ be a future defect of $\partial^h(-)$ not passed up to $\partial(+)$, 
let $\x_0\wedge\x_1\prec x\in I$ so $\F\psi\in h(\x)$.  Since $h\restr{[\x_0\wedge\x_1, \x_0\vee\x_1]}$ is a rectangle model  there is $\y\geq\x\in [\x_0\wedge\x_1, \x_0\vee\x_1]$ with $\psi\in h(\y)$.  Since $\F\psi$ is a defect of $I$ we cannot have $\y\in I$ and since the defect is not passed up to $\partial(+)$ we cannot have $\y\in J$ either, hence $\y\in R(S)$ for some $S\in E$.  Observe that $\F\psi$ holds at all points below $\y$, hence the defect $\F\psi$ is passed up to the bottom corner of the rectangle $R(S)$ if it is proper, or to the  singleton holding at the point rectangle $R(S)$ otherwise.  Condition~\ref{sh:3} holds because $h$ is a rectangle model: if $\P\psi$ is a defect of $h(\x)$ where $\x$ is either the bottom corner of a rectangle $R(S)$ or the single point of $R(S)$ then there must be $\y< \x$ with $\P\psi\in h(\y)$,  and as we noted above, $\y\in I$, so the defect is passed down to $\partial(-)$. For condition~\ref{sh:4}, since $\x_0$ has no $\triangleleft$ successor and $\x_1$ has no $\triangleright$ successor, and since all points  $\z\in P$ with $\x_0\triangleleft\z\triangleleft \x_1$ are in the interior of $[\x_0\wedge\x_1, \x_0\vee\x_1]$, it follows that the lower edges of this rectangle are covered by $\partial(-)$ hence every future defect of these lower edges is passed up to $\partial(-)$, similarly the upper edges of the rectangle are covered by $\partial(+)$ and their past defects are passed down to $\partial(+)$.

\medskip

From claim~\ref{cl:E} $R(S)$ is good, for all $S\in E$.  By claim~\ref{cl:singleton} either $E$ consists of a single equivalence class $S$ and $R(S)$ is good by claim~\ref{cl:E},  or $|E|>1$.  In the latter case we have $\partial(h,  [\x_0\wedge\x_1\wedge\x_0\vee\x_1])$ is a shuffle  of boundary maps $\partial(h, R(S))$,  where $S$ is an $\approx$-equivalence class.   Since $B$ is closed under shuffles it follows that $\partial(h, [\x_0\wedge\x_1, \x_0\vee\x_1])\in B$.  But $h$ is the join of its restrictions to $ [\x_0\wedge\x_1, \x_0\vee\x_1]$, \/  $[\x_0\wedge\r, \x_1],\; [\x_0, \l\vee\x_1],\; W(\x_0)$ and $E(\x_1)$ (see figure~\ref{fig:EW}).  By claims~\ref{cl:E} and \ref{cl:join} these latter four rectangles are  good, hence $\partial = \partial(h, \dom(h)) \in B$, proving the induction step.  Though not required, the same argument shows that $\partial(h, R)\in B$ for any proper subrectangle $R\subseteq \dom(h)$, so $\dom(h)$ is good. 

\end{proof}

\begin{lemma}\label{lem:4square}
$\phi$ is satisfiable over $(\reals^2, <)$ iff there are boundary maps $\partial_1, \partial_2, \partial_3, \partial_4\in B$ such that $(\partial_1\oplus_E\partial_2)\oplus_N(\partial_3\oplus_E\partial_4)$ is defined (the boundary maps fit together) and open, and $\phi\in\partial_1(\t) \; (=\partial_2(\l)=\partial_3(\r)=\partial_4(\b))$.
\end{lemma}
\begin{proof}  Suppose $\phi$ is satisfiable, say $(\reals^2, <, v), \x\models\phi$  for some valuation $v$, some $\x\in\reals^2$.  By lemma~\ref{lem:truth} the map $h_v:\reals^2\rightarrow \MCS$ defined by $h_v(\x)=\set{\psi\in Cl(\phi): (\reals^2, <, v), \x\models\psi}$ is an open rectangle model.  Let  $\partial_1=\partial(h_v, S(\x)\cap W(x)), \;\partial_2=\partial(h_v, S(\x)\cap E(\x)),\; \partial_3=\partial(h_v, N(\x)\cap W(\x))$ and $\partial_4=\partial(h_v, N(\x)\cap E(\x))$.  The right hand side of the lemma holds, by lemma~\ref{lem:split}.

Conversely, if the right hand side of the lemma holds then by lemma~\ref{lem:main} there are rectangle models $h_1, h_2, h_3, h_4$ such that $\partial_i=\partial^{h_i}$, for $i=1,2,3,4$.  Since the boundary maps fit, we can use lemma~\ref{lem:equiv} and find rectangle models $h_1', h_2', h_3', h_4'$  such that $\partial_i=\partial^{h'_i}$, where $\dom(h_1')=\set{\x\in\reals^2:\x\leq(0,0)}$, etc. and these rectangle models agree with each other on their common boundaries.  Hence $h=h_1'\cup h_2'\cup h_3'\cup h_4'$ is a well-defined  rectangle model and $\phi\in h(0,0)$.   By lemma~\ref{lem:truth}, $(\reals^2, <, v_h), (0,0)\models\phi$.
\end{proof}

\begin{theorem}\label{thm:decidable}
Validity over $(\reals^2, <)$ is decidable.  Validity over $(\reals^2, \leq)$ is also decidable.
\end{theorem}
\begin{proof}  We show that satisfiability is decidable (hence validity is too).  Let $\phi$ be a temporal formula.     By lemma~\ref{lem:4square},  satisfiability of $\phi$ is equivalent to the existence of four matching proper boundary maps in $B$ with a common corner labelled by some $m\ni\phi$.  Since $B$ is recursive and we can easily test whether proper boundary maps match, we get our decidability result.

 As we mentioned earlier, to check if $\phi$ is satisfiable over $(\reals^2, \leq)$ it suffices to check if $\phi'$ is satisfiable over $(\reals^2, <)$, where $\phi'$ is obtained from $\phi$ be replacing any subformula $\F\psi$ of $\phi$ by $(\psi\vee\F\psi)$ and any subformula $\P\psi$ by $(\psi\vee\P\psi)$.   
\end{proof}

\section{Complexity}\label{sec:complexity}
Let the length of $\phi$ be $n$.  Its not hard to check that the numbers of \MCSss, traces and boundary maps are bound above by an exponential function of $n$.  Hence the number of runs of  the while loop in algorithm~\ref{alg} is bound by an exponential function, and each iteration runs in exponential time, hence the run-time of the algorithm is bound by an exponential function of $n$.  We tighten that upper complexity bound slightly by devising a non-deterministic algorithm to check if a proper boundary map belongs to $B$ using polynomial space, thereby showing that membership of $B$ and hence also validity over  $(\reals^2, <)$ is in PSPACE. 

\begin{algorithm}
\begin{algorithmic}
\Procedure{$B$}{$\partial$}
\Switch
\Case{0}
 \State Check  $\partial\in B_0$ \Comment{Uses space $\delta=|\partial|$}
\EndCase
\Case{1}
\State Pick $\partial_1, \partial_2$ and $d\in\set{N, S, E, W}$
\State Check  $\partial=\partial_1\oplus_d \partial_2$\Comment{Release $\partial, d$}
\State  Check  $B(\partial_1)$\Comment{Release $\partial_1$}
\State Check  $B(\partial_2)$
\EndCase

\Case{2}
\State Pick $\partial'$ and $d\in\set{N, S, E, W}$
\State Check $\partial'\stackrel{d}\rightarrow\partial$ \Comment{Release $\partial$}
\If {$d\in\set{E, W}$} Check $VD(\partial')$ 
 \Else  $\;$ Check $HD(\partial')$ 
 \EndIf
\EndCase
\Case{3a} 
\State Pick $\partial'$, pick $m\in \MCS$
\State Check  $\partial$ is a shuffle of $\set{\partial', \partial_m}$ \Comment{Release $\partial, m$}
\State Check $\partial'=\partial_{m'}$ (some $m'\in'MCS$) or $B(\partial')$
\EndCase
\Case{3b}
\State Pick $m\in \MCS$, let $S=\set {\partial_m}$
\For {$i<n$}  
\State Pick $\partial_i$ and add to $S$\Comment{Needs $n\cdot\delta$ space, plus}
\State If $\partial_i$ is proper, check that $B(\partial_i)$\Comment{space for one $B(\partial_i)$}
\EndFor
\State Check  $\partial$ is a shuffle of $S$
\EndCase
\EndSwitch
\EndProcedure
\Procedure{$VD$}{$\partial$}\Comment{Check that $\partial$ has vert. displaced rect. model}
\State Check $\partial(W)=\partial(E)$, let $\partial(W)=(c_0, m_0, \ldots, c_p)$
\Switch
\Case{a}
\For{$i=1$ to $p-1$}
\State Pick $\partial_i^-, \partial_i^+$, check $B(\partial_i^-), \; B(\partial_i^+)$
\State  Check $Split(\partial, horiz, i, \partial_i^-, \partial_i^+)$\Comment{Release $\partial_i^-, \partial_i^+$}
\EndFor
\EndCase
\Case{b}
\State Pick $\partial_0, \partial_{1+2}$, check $\partial=\partial_0\oplus_N\partial_{1+2}$\Comment{Release $\partial$}
\State Check $B(\partial_0)$\Comment{Release $\partial_0$}
\State Pick $\partial_1, \partial_2$ check $\partial_{1+2}=\partial_1\oplus_N\partial_2$\Comment{Release $\partial_{1+2}$}
\State Check $B(\partial_2)$\Comment{Release $\partial_2$}   
\State Check $\partial_1(W)=(c_0),\; \partial_1(E)=(c_p)$
\State Check $B(\partial_1)$
\EndCase
\EndSwitch
\EndProcedure
\Procedure{$HD$}{$\partial$}
Similar
\EndProcedure
\end{algorithmic}
\caption{\label{alg2} Algorithm to check that $\partial\in B$}
\end{algorithm}

\begin{lemma}\label{lem:alg2}
The non-deterministic algorithm $B(\partial)$ shown in algorithm~\ref{alg2} correctly determines whether there is a rectangle model $h$ such that $\partial=\partial^h$.  For any rectangle model $h$, there is a run of $B(\partial^h)$ using only polynomial space.
\end{lemma}
\begin{proof}
We prove, by induction over the number of recursive calls to $B(),\;VD()$ or $HD()$ in a  run of algorithm~\ref{alg2}, that if there is a successful run of $B(\partial)$ then there is a rectangle model $h$ such that $\partial=\partial^h$,  and  if there is a successful run of $VD(\partial)$  then there is a vertically displaced rectangle model $g$ such that $\partial=\partial^g$ (and a similar result for $HD()$).  A successful run of any of these procedures with no recursive procedure calls only happens when $\partial\in B_0\subseteq B$, and every $\partial\in B_0$ has a vertically and horizontally displaced rectangle model.   For the general case, if there is a successful run of $B(\partial)$ then it is either in $B_0$, a join, a limit or a shuffle of boundary maps of rectangle models (according to which option is selected initially) and if it is a $d$-limit of $\partial'$ then $\partial'$ has a vertically/horizontally displaced rectangle model (depending on the direction $d$), so $\partial$ has a rectangle model by lemma~\ref{lem:joins}.  If there is a successful run of $VD(\partial)$ selecting  (a)  then $\partial$ has a vertically displaced rectangle model (use lemma~\ref{lem:displaced}(\ref{cond:2}$\Rightarrow$\ref{cond:1}).  If a successful run of $VD(\partial)$ selects (b) then $\partial_1(W)=(c_0),\;\partial_1(E)=c_k$, so either $k=0$ but then $\partial\in B_0$, or $k>0$, \/ $\partial_1, \partial_2, \partial_3$ each have rectangle model, by lemma~\ref{lem:joins} $\partial$ has a rectangle model and since $k>0$ any rectangle model for $\partial$ is vertically displaced.

The hard part is to prove that if $h$ is a rectangle model then there is an accepting run of $B(\partial^h)$ in algorithm~\ref{alg2} using only polynomial space.  Let $\delta$ be enough space to store any boundary map.  The length of a subformula, the number of formulas in an \MCS and the number of clusters/\!\!\!\MCSs in a trace are each no more than $n$, the length of $\phi$, hence $\delta=O(n^3)$.    Define the function $S:\nats\to\nats$  by $S(k)= \delta\times\max\set{5, n} + \log_2(n)+S(k-1)$ and $S(0)=\delta$, so for $n\geq 5$ we have $S(k)   =  \delta\cdot n\cdot k+\log_2n\cdot k+\delta$.  We claim that if $h$ has height at most $k$ then there is an accepting run of $B(\partial^h)$ using at most $S(k)$ space.  To help with this proof we introduce some extra notation.  In a run of $B(\partial^h)$ we may say that $(1, R_1, R_2, d)$ is chosen if in the first step option 1 is selected, proper boundary maps $\partial(h, R_1), \partial(h, R_2)$ are picked ($R_1, R_2$ are rectangles) and direction $d\in\set{N, S, E, W}$ is chosen.  Similarly, we may write $(0), (2, R, d), (3a, m, R)$ or  $(3b, m, (R_i:i<n))$  to indicate that one of the other options is chosen.  If $B(\partial)$ has a successful run we write $\rho(B(\partial))$ to denote the minimum, over all possible successful runs of $B(\partial)$, of the space used in a successful run.  Similarly, we write $\rho(VD(\partial))$ for the minimum over possible successful runs of $VD(\partial)$ of the space used in a run. 

For $k=0$, if $h$ has height zero then $\partial^h\in B_0$ and there is a successful run of $B(\partial^h)$ where option $(0)$ is chosen, using only the space needed to store $\partial^h$, i.e. $S(0)=\delta$.  Now let $h$ have  height $k>0$.     Let $I, J,  
\Gamma, \Delta, P, \approx$ be as defined just before claim~\ref{cl:int} and illustrated in figure~\ref{fig:gamma}.   Our induction hypothesis is that for any rectangle $R$ if the height of $\partial(h, R)$ is strictly less than $k$ then $\rho(B(\partial(h, R)))\leq S(k-1)$ and if $h\restr{R}$ is also vertically (horizontally) displaced then $\rho(VD(\partial(h, R)))\leq S(k-1)$ (respectively, $\rho(HD(\partial(h, R)))\leq S(k-1)$).

Suppose $h\restr{R}$ is vertically displaced and has height $k$.
By lemma~\ref{lem:displaced}(\ref{cond:1}$\Rightarrow$\ref{cond:3}) there is a successful run of $VD(\partial(h, R))$
 which either selects option (a) first and chooses, for each $i<p$ in turn,  $\partial_i^-, \partial_i^+$ with height strictly less than $k$, so $\rho(VD(\partial(h, R)))$ is at most  $S(k-1)$ to run $B(\partial_i^-)$ plus $2\delta$ to store $\partial_i^+, \partial$ plus $\log_2n$ to store $i$, or it selects (b) initially and chooses $\partial_0, \partial_2$ of height strictly less than $k$, uses at most $S(k-1)+2\delta$ space to check $B(\partial_0)$ while storing $\partial_2, \partial$, then uses at most $S(k-1)+\delta$ space to  run $B(\partial_2)$ while storing $\partial_1$,  and finally runs $B(\partial_1)$.  Thus, 
\begin{equation}\label{eq:VD}
\rho(VD(\partial(h, R)))\leq max(2\delta+\log_2n+S(k-1),\rho(B(\partial(h, S))):S\subseteq R).
\end{equation}
Hence, if we can prove that $\rho(B(\partial(h, S)))\leq S(k)$, for all rectangle models of height at most $k$, it will follow that $\rho(VD(\partial(h, R)))\leq S(k)$ for vertically displaced rectangle models of height at most $k$.

\medskip

Before focussing on the main procedure $B(\partial^h)$, an observation.
 By considering a run of $B(\partial(h, R))$ that selects $(1,  R\cap S(\x), R\cap N(\x))$ initially (for some $\x$ in the interior of $\dom(h)$, so $R\cap S(\x), R\cap N(\x)$ are proper rectangles), runs $B(\partial(h, R\cap S(\x)))$ while storing $\partial(h,  R\cap N(\x))$ and then runs $B(\partial(h,  N(\x)))$,  we see  that
 \begin{equation}\label{eq:rho} \rho(B(\partial(h, R)))\leq max(\delta+\rho(B, \partial(h,  R\cap S(\x))), \rho(B(\partial(h,  R\cap N(\x)))))
 \end{equation}
and similar upper bounds can be obtained corresponding to the other three directions.

Now for the  main procedure $B(\partial^h)$.
 First suppose that $\Gamma\cap\Delta=\emptyset$. 
  We will show that repeated selection of option 1  (and calls to $B(\partial')$ where $\partial'$ has height strictly less than $k$) can  provide an accepting run of $B(\partial^h)$ using only $\delta+S(k-1)$ space.  By lemma~\ref{lem:apart} there is $\epsilon>0$ and the distance between a point in $\Gamma$ and a point in $\Delta$ is always more than $\epsilon$, see the first part of figure~\ref{fig:epsilon}.   Hence, either $\Gamma$ meets the Western edge of $\overline{\dom(h)}$ at a point $\x$ at least $\frac\epsilon{\sqrt{2}}$ South of $\l$, or $\Delta$ meets the Northern edge of the rectangle at least $\frac\epsilon{\sqrt{2}}$ to the East of $\l$, without loss assume the former.  By \eqref{eq:rho}, $\rho(B(\partial^h))\leq max(\delta+\rho(B(\partial(h, N(\x)))), \rho(B(\partial(h, S(\x)))))\leq max(\delta+S(k-1), \rho(B(\partial(h, S(\x))))$, since the height of $\partial(h, N(\x))$ is strictly less than $k$.  Now $h\restr{S(\x)}$ is a restriction of $h$ to a rectangle where one of the dimensions has been reduced by at least $\frac\epsilon{\sqrt2}$.  So we may continue to apply \eqref{eq:rho} until we reduce to a restriction of $h$ to a rectangle map of height strictly less than $k$.  Hence $\rho(B(\partial^h))\leq\delta+S(k-1)$.

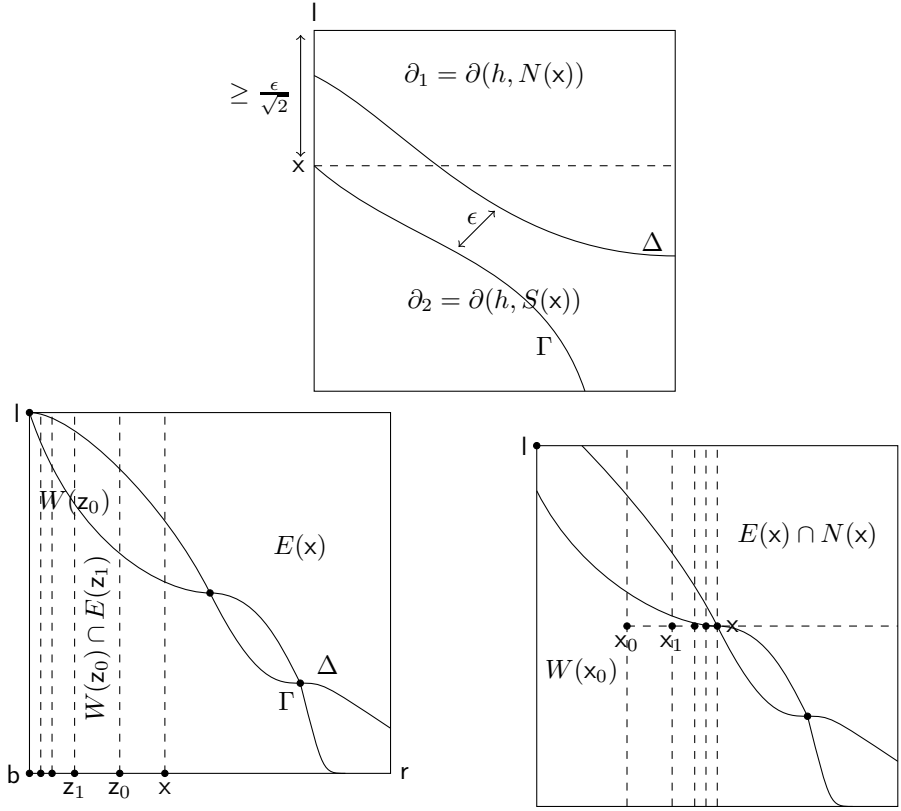
\begin{figure}
\begin{tikzpicture}[scale=.6]

\node[left, above] at (0,8) {$\l$};
\draw (0,0) rectangle (8,8);

\draw  ((6,0) .. controls (5,3)  and (2,3.2)   .. (0,5);

\draw (0,7) .. controls (2,6) and (4,3) .. (8,3);



\node at (5.1,1) {$\Gamma$};

\node at (7.5, 3.3) {$\Delta$};

\draw [<->] (3.2,3.2)--(4,4);
\node [above] at (3.5,3.5){$\epsilon$};

\draw[dashed](0,5.0)--(8,5.0);
\node[left] at (0,5) {$\x$};
\draw[<->](-.3,5.2)--(-.3,7.9);
\node[left] at (-.3, 6.5) {$\geq \frac\epsilon{\sqrt 2}$};

\node at (4,7) {$\partial_1=\partial(h, N(\x))$};

\node at (4,2) {$\partial_2=\partial(h, S(\x))$};

\end{tikzpicture}
\\
\begin{tikzpicture}[scale=.6]

\draw (0,0) rectangle (8,8);

\draw (0, 8) .. controls (1,5) and (3,4) ..(4,4);
\draw (4,4) .. controls (5, 4)  and (5.5,3).. (6,2);
\draw (6,2) .. controls (6.5,0).. (7,0);

\draw(0,8) .. controls (1, 8) and (3, 6) .. (4,4);
\draw(4,4) .. controls (5, 2) and (5.5, 2) ..(6,2);
\draw(6,2) .. controls(6.5, 2).. (8,1);

\draw[fill] (0,8) circle [radius=2pt];
\draw[fill] (4,4) circle [radius=2pt];
\draw[fill] (6,2) circle [radius=2pt];

\draw[fill] (0,0) circle [radius=2pt];

\draw[fill] (2,0) circle [radius=2pt];
\draw[fill] (1,0) circle [radius=2pt];
\draw[fill] (0.5,0) circle [radius=2pt];
\draw[fill] (0.25,0) circle [radius=2pt];
\draw[fill] (3,0) circle [radius=2pt];

\draw[dashed](3,0)--(3,8);
\draw[dashed](2,0)--(2,8);
\draw[dashed](1,0)--(1,8);
\draw[dashed](.5,0)--(.5,8);
\draw[dashed](.25,0)--(.25,8);

\node at (5.7,1.6) {$\Gamma$};

\node at (6.6,2.4) {$\Delta$};

\node[below] at (3,0){$\x$};

\node[left] at (0,8) {$\l$};
\node[left] at (0,0) {$\b$};
\node[right] at (8,0){$\r$};

\node[below] at (2,0){$\z_0$};
\node[below] at (1, 0) {$\z_1$};

\node at (6,5) {$E(\x)$};

\node[rotate=90] at (1.5,3) {$W(\z_0)\cap E(\z_1)$};

\node at (1, 6) {$W(\z_0)$};

\end{tikzpicture}
\hspace{.4in}
\begin{tikzpicture}[scale=.6]

\draw (0,0) rectangle (8,8);

\draw (0, 7) .. controls (1,5) and (3,4) ..(4,4);
\draw (4,4) .. controls (5, 4)  and (5.5,3).. (6,2);
\draw (6,2) .. controls (6.5,0).. (7,0);

\draw(1,8) .. controls (1, 8) and (3, 6) .. (4,4);
\draw(4,4) .. controls (5, 2) and (5.5, 2) ..(6,2);
\draw(6,2) .. controls(6.5, 2).. (8,1);

\draw[fill] (0,8) circle [radius=2pt];
\draw[fill] (4,4) circle [radius=2pt];
\draw[fill] (6,2) circle [radius=2pt];

\draw[fill] (2,4) circle [radius=2pt];
\draw[fill] (3,4) circle [radius=2pt];
\draw[fill] (3.5,4) circle [radius=2pt];
\draw[fill] (3.75,4) circle [radius=2pt];

\draw[dashed](4,0)--(4,8);

\draw[dashed](2,0)--(2,8);
\draw[dashed](3, 0)--(3, 8);
\draw[dashed](3.5, 0)--(3.5, 8);
\draw[dashed](3.75, 0)--(3.75, 8);
\draw[dashed](2,4)--(8,4);

\node[left] at (0, 8) {$\l$};

\node[above,right] at (4,4){$\x$};

\node at (1,3) {$W(\x_0)$};

\node[below] at (2,4){$\x_0$};
\node[below] at (3,4){$\x_1$};
\node at (6,6) {$E(\x)\cap N(\x)$};
\end{tikzpicture}

\caption{\label{fig:epsilon}$\Gamma\cap\Delta=\emptyset$, \/ $\l\in \Gamma\cap\Delta$  or $\l\not\in\Gamma\cap\Delta$, finite}
\end{figure}

We will show next that if $ \Gamma\cap\Delta$ is finite then there is a successful run of $B(\partial^h)$  selecting options 1 and 2 using at most $3\delta+S(k-1)$ space.   The proof is by induction over the size of $\Gamma\cap\Delta$,   we have already established this when $\Gamma\cap\Delta=\emptyset$ (with room to spare), so now we assume the result for cases where $\Gamma\cap\Delta$ is strictly smaller. For the inductive step, first suppose $\Gamma\cap\Delta$ is finite and $\l\in\Gamma\cap\Delta$, let $\x\in\dom(h)$ be any point East of $\l$ and West of any other member of $\Gamma\cap\Delta$, see the second part of figure~\ref{fig:epsilon} for an illustration.  
If the Western edge of $h$ is not simple, then let $\w$ be the supremum of the set $Z$ of points $\z$ on the Western edge $(\b, \l)$ such that $h\restr{ S(\z)}$ has a simple Western edge.  By claim~\ref{cl:simple},  $\w\in\Gamma,\; \w\in Z$ and  $h\restr{S(\z)}$  has a simple Western edge.  Since the height of $\partial(h, N(\w))$ is less than $k$, by \eqref{eq:rho} we get  $\rho(B(\partial^h))\leq max(\delta+S(k-1), \rho(B(\partial(h, S(\w)))))$ and $\partial(h, S(\w))$ has a simple Western edge.  So we may assume that $h$ has a simple Western edge.
 Use Ramsay's theorem to find an infinite sequence $\z_0>_1\z_1>_1\ldots$ converging to $\b$ from the East, all in $W(\x)$, where $\partial(h, W(\z_j)\cap E(\z_i))$ is constant (for $i<j$).  We  apply lemma~\ref{lem:limit3} to show that  $\partial(h, W(\z_0)\cap E(\z_1))$ has a vertically displaced rectangle model, $\partial(h, W(\z_0)\cap E(\z_1)) \stackrel{W}{\rightarrow}\partial(h, W(\z_0))$ and $\partial(h, W(\z_0))$ is a Western limit.   Since  $\z_0, \z_1\in W(\x) $ we know that $\Gamma\cap\Delta$ does not meet $W(\z_0)\cap E(\z_1)$.  So a successful run of $B(\partial^h)$ goes like this.
\[\begin{array}{lcl}
B(\partial^h)&\rightarrow& \move{W(\z_0)}{E(\z_0)}{E}\\
B(\partial(h, W(\z_0)))&\rightarrow &\movel{W(\z_0)\cap E(\z_1)}{W}\\
VD(\partial(h, W(\z_0)\cap E(\z_1)))&&\\
B(\partial(h, E(\z_0)))&&
\end{array}
\]
In the first line the run of $B(\partial^h)$ selects option 1 and boundary maps $\partial(h, W(\z_0))$, $\partial(h, E(\z_0))$ and direction $E$,  and in the second line option 2  is selected  for the recursive call to $B(\partial(h, W(\z_0)))$, while storing $\partial(h, E(\z_0))$.  By the previous case (where $\Gamma\cap\Delta=\emptyset$) for any subrectangle $R$ of $W(\z_0)\cap E(\z_1)$ we have $\rho(B(\partial(h, R)))\leq\delta+S(k-1)$, so
by \eqref{eq:VD}, the space needed to run $VD(\partial(h, W(\z_0)\cap E(\z_1)))$ in the third line is not more than $2\delta+S(k-1)$, plus $\delta$ to store $\partial(h, E(\z_0))$, hence the space needed is at most $S(k)$.

For the case when $\l\not\in\Gamma\cap\Delta$ (and $\Gamma\cap\Delta$ is finite), let $\x$ be the leftmost point of $\Gamma\cap\Delta$, see the third part of figure~\ref{fig:epsilon}.    Find a sequence  $\x_0, \x_1, \ldots$  converging from the West to $\x$ such that $\partial(h, W(\x_i)\cap E(\x_j)\cap N(\x))$ is constant, for $i<j$.   $E(\x_0)\cap W(\x)\cap N(\x)$ has a simple Eastern edge (by claim~\ref{cl:simple}), so we may apply lemma~\ref{lem:limit3} to show  that   $\partial(h, E(\x_0)\cap W(\x_1)\cap N(\x))$ has a vertically displaced rectangle model and $\partial(h, E(\x_0)\cap W(\x)\cap N(\x))$ is an Eastern limit of it.
A successful run of $B(\partial^h)$ goes like this.
\[\begin{array}{lcl}
B(\partial^h)&\rightarrow& \move{W(\x_0)}{E(\x_0)}{E}\\
B(\partial(h, W(\x_0)))&&\\
B(\partial(h, E(\x_0)))&\rightarrow&\move{E(\x_0)\cap W(\x)}{E(\x)}{E}\\
B(\partial(h, E(\x_0)\cap W(\x)))&\rightarrow& \move{E(\x_0)\cap W(\x)\cap S(\x)}{E(\x_0)\cap W(\x)\cap N(\x)}{N}\\ 
B(\partial(h, E(\x_0)\cap W(\x)\cap S(\x)))&\rightarrow&(0)\\
B(\partial(h, E(\x_0)\cap W(\x)\cap N(\x)))&\rightarrow&\movel{E(\x_0)\cap W(\x_1)\cap N(\x)}{E}\\
VD(\partial(h, E(\x_0)\cap W(\x_1)\cap N(\x)))&&\\
B(\partial(h, E(\x)))&\rightarrow&\move{E(\x)\cap N(\x)}{E(\x)\cap S(\x)}{S}\\
B(\partial(h, E(\x)\cap N(\x)))&\rightarrow&(0)\\
B(\partial(h, E(\x)\cap S(\x))).&&
\end{array}
\]
The second line requires $\delta+S(k-1)$ space (by the case $\Gamma\cap\Delta=\emptyset$) plus  $\delta$ to store $\partial(h, E(\x_0))$.  The fifth  line requires  $\delta$ (since the height of $\partial(h, E(\x_0)\cap W(\x)\cap S(\x))$ is zero) plus $2\delta$ to store $\partial(h, E(\x))$ and $\partial(h, E(\x_0)\cap W(\x)\cap N(\x))$.   Line 7 
 requires $2\delta+S(k-1)$ space (by \eqref{eq:VD} and $\Gamma\cap\Delta=\emptyset$ case) plus $\delta$ to store $\partial(h, E(\x))$.  The final line requires $3\delta+S(k-1)$ space, by the previous case, since $\x\in\Gamma\cap\Delta$ is the left corner of $E(\x)\cap S(\x)$. This proves our result for the case where $\Gamma\cap\Delta$ is finite.

Next we show that if $\x, \y\in\Gamma\cap\Delta$ and $\x\approx\y$ then there is a successful run of $B(\partial(h, E(\x)\cap W(\y))$ using at most $4\delta+S(k-1)$ space.  We do this by letting $\x$ be the leftmost point of some $\approx$-equivalence class and showing that the set of $\y\approx\x$ for which there is a successful run of $B(\partial(h, E(\x)\cap W(\y))$ using that amount of space is closed under successors and closed under limits.   So  suppose $\x\triangleleft\y\in\Gamma\cap\Delta,\;\x\approx\y$ and  that there is a run of $B(\partial(h,E(\x)\cap W(\y)))$ using only $4\delta+S(k-1)$ space.  Let $\y^+\in\Gamma\cap\Delta$ be an immediate successor of $\y$.  Here is a successful run of $B(\partial(h, E(\x)\cap W(\y^+)))$
\[\begin{array}{lcl}
B(\partial(h, E(\x)\cap W(\y^+)))&\rightarrow&\move{E(\y)\cap W(\y^+)}{E(\x)\cap W(\y)}{W}\\
B(\partial(h, E(\y)\cap W(\y^+)))&&\\
B(\partial(h, E(\x)\cap W(\y))).&&
\end{array}
\]
The second line requires only  $3\delta+S(k-1)$ space (by previous case where $\Gamma\cap\Delta$ is finite) plus $\delta$ to store $\partial(h, E(\x)\cap W(\y))$.
The final line requires at most $4\delta+S(k-1)$ space, by assumption.

Now suppose $\x\approx \y\in\Gamma\cap\Delta$,\/ $\y$ is not a successor and for all $\z\in\Gamma\cap\Delta$ with $\x\triangleleft\z\triangleleft\y$   we have $\rho(B(\partial(h,E(\x)\cap W(\z))))\leq4\delta+S(k-1)$, see figure~\ref{fig:y}.   We must show that there is a successful run of $B(\partial(h, E(\x)\cap W(\y)))$ using that amount of space.     
 We start by reducing to the rectangle  $R=E(\x)\cap S(\x)\cap W(\y)\cap N(\y) = [\x\wedge\y, \x\vee\y]$.  
 By two applications of \eqref{eq:rho} we get
\[\rho(B(\partial(h, E(\x)\cap W(\y))))\leq max(\delta +S(k-1),  \rho(B(\partial(h, R))))\]
since the heights of $h\restr{N(\x)\cap E(\x)}, h\restr{W(\y)\cap S(\y)}$ are strictly less than $k$.  It remains to show that  $\rho(B(\partial(h, R)))\leq 4\delta+S(k-1)$.

By claim~\ref{cl:seq}, for each $\epsilon>0$ there is an $\omega$-sequence or an $\overline{\omega}$-sequence whose range is to the left of $\y$ but within $\epsilon$ of it.  
For each $\z\in[\x\wedge\y, \x\vee\y]$ let $\trace^v(h, \z)$ be the extended trace of the closed vertical line segment $[\z\wedge\y, \z\vee\x]$ (see definition~\ref{def:trace}).  Since $\y$ is not a successor, $\tau^v(\y)$ has a simple Eastern edge (by claim~\ref{cl:simple}).  Since the number of extended traces is finite there must be a single extended trace $\trace$ such that for all $\epsilon>0$ there are distinct $\z, \z'\in \Gamma\cap\Delta$ within $\epsilon$ of $\y$ belonging to the range of an $\omega$-sequence or an $\overline{\omega}$-sequence, and where $\trace^v(h, \z)=\trace^v(h, \z')=\trace$.   Since they are in the range of the same $\omega$-sequence or $\overline\omega$-sequence, there are only finitely many points in $\Gamma\cap\Delta\cap [\z\wedge\z',\z\vee\z']$.  Let $\z_0, \z_1, \ldots$ be a sequence converging from the left to $\y$ such that $\trace^v(h, \z_i)=\trace$ (all $i$).  Observe that $\partial(h, E(\z_i)\cap W(\z_j)\cap N(\y)\cap S(\x))$ is constant (for $i<j$) so by lemma~\ref{lem:limit3} $\partial(h, E(\z_0)\cap R)$ is an Eastern limit of it.  So a successful run of $B(\partial(h, R))$ is:
\[\begin{array}{lcl}
B(\partial(h, R)&\rightarrow&\move{R\cap E(\z_0)}{R\cap W(\z_0)}{W}\\
B(\partial(h, R\cap E(\z_0)))&\rightarrow & \movel{R\cap E(\z_0)\cap W(\z_1)}{E}\\
VD(\partial(h, R\cap E(\z_0)\cap W(\z_1)))&&\\
B(\partial(h, R\cap W(\z_0))).&&
\end{array}\]
The third line uses space $3\delta+S(k-1)$ space (by \eqref{eq:VD} and by the case where $\Gamma\cap\Delta$ is finite, since $\partial(h, R\cap E(\z_0)\cap W(\z_1))=\partial(h, R\cap E(\z)\cap W(\z'))$) plus $\delta$ to store $\partial(h, R\cap W(\z_0))$ and the last line uses space $4\delta+S(k-1)$, by our assumption.
We conclude that $\x\approx\y$ implies there is a successful run of $B(\partial(h, E(\x)\cap W(\y)))$ using only $4\delta+S(k-1)$ space.
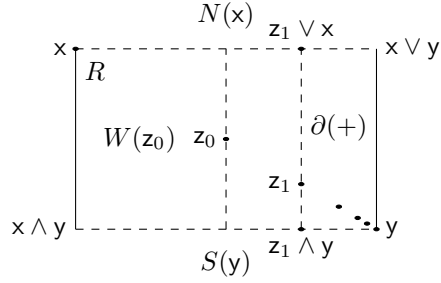
\begin{figure}

\begin{tikzpicture}[xscale=.5,yscale=.3]

\node[left] at (0,8) {$\x$};
\node[right] at (8,0){$\y$};
\node[left] at (0,0){$\x\wedge\y$};
\node[right] at (8,8){$\x\vee\y$};
\draw (0,0)--(0,8);
\draw (8,0)--(8,8);
\draw[dashed](0,0)--(8,0);
\draw[dashed](0,8)--(8,8);
\node at (4, 9.5) {$N(\x)$};
\node at (4, -1.5){$S(\y)$};

\node at (7,4.5) {$\partial(+)$};

\node[below] at (6,0) {$\z_1\wedge\y$};
\node[above] at (6,8) {$\z_1\vee\x$};

\draw[fill] (6,0) circle [radius=2pt];
\draw[fill] (6,8) circle [radius=2pt];

\node at (1.7,4) {$W(\z_0)$};

\node[right] at(0,7){$R$};

\draw[fill] (8,0) circle [radius=2pt];
\draw[fill] (0,8) circle [radius=2pt];

\draw[fill] (4,4) circle [radius=2pt];
\draw[fill] (6,2) circle [radius=2pt];
\draw[fill] (7,1) circle [radius=2pt];
\draw[fill] (7.5,.5) circle [radius=2pt];
\draw[fill] (7.75,.25) circle [radius=2pt];

\draw[dashed](4,0)--(4,8);
\draw[dashed](6,0)--(6,8);

\node[left] at (4,4){$\z_0$};
\node[left] at (6,2) {$\z_1$};

\end{tikzpicture}
\caption{\label{fig:y} $\y$ is not a successor}
\end{figure}

Finally, we prove the main lemma by showing that if $h$ has height $k$ then there is a successful run of $B(\partial^h)$ using at most $S(k)$ space.  If there is only one $\approx$-equivalence class then there is a successful run of $B(\partial^h)$ using only $4\delta+S(k-1)$ space, as we have seen.  Suppose there is more than one $\approx$-equivalence class.   Let $S_0, S_1$ be the $\approx$-equivalence classes of $\l, \r$, respectively, and let $\x_0=\r(S_0),\;\x_1=\l(S_1)$, see figure~\ref{fig:EW}.  A successful run of $B(\partial^h)$ will first reduce to the problem of running $B(\partial(h,   [\x_0\wedge\x_1, \x_0\vee\x_1]))$, using \eqref{eq:rho} four times and removing $W(\x_0),   E(\x_1),\; [\x_0\wedge\r, \x_1]$ and $[\x_0, \l\vee \x_1]$ from the domain of $h$ (see figure~\ref{fig:EW}).  By \eqref{eq:rho} and the foregoing, the space required is at most $\delta \; + \; (4\delta+S(k-1))$, or $\rho(B(\partial(h, [\x_0\wedge\x_1, \x_0\vee\x_1])))$ if greater.
%
 By claim~\ref{cl:shuffle} $\partial(h, [\x_0\wedge\x_1, \x_0\vee\x_1])$ is a shuffle of boundary maps  $\set{\partial(h, R(S)):S\in E\setminus\set{S_0, S_1}}$.    By lemma~\ref{lem:shuffles}, it is either a shuffle of a  subset $\set{\partial(h, R(S)), \partial_{h(\x)}}$ (some  $S, \set{\x}\in E\setminus\set{S_0, S_1}$) or it is a shuffle of $n$ boundary maps $\set{\partial(h, R(S_i)): i<n}$ from this shuffle set, where each $\partial(h, R(S_i))$ has height strictly less than $k$.  In the former case there is a successful run of $B(\partial(h, [\x_0\wedge\x_1, \x_0\vee\x_1]))$ that selects $(3a, h(\x),  R(S))$, checks that $\partial$ is a shuffle of $\set{\partial(h, R(S)), \partial_{h(\x)}}$, using space $3\delta$,  and then runs $B(\partial(h, R(S)))$ (if $R(S)$ is proper) using only $4\delta+S(k-1)$ space (by the previous case $\x, \y\in\Gamma\cap\Delta,\; \x\approx\y$).   In the latter case,
there is a successful run of $B(\partial(h,  [\x_0\wedge\x_1, \x_0\vee\x_1]))$ which selects  $(3b, \set{R(S_i):i<n)}$ initially,  solves each $B(\partial(h, R(S_i)))$, one at a time using space $S(k-1)$, plus $n\delta$ to store $\set{\partial(h, R(S_i)):i<n}$ and $\log_2n$ to store $i$, and then checks that $\partial$ is a shuffle of $\set{\partial(h, R(S_i)):i<n}$ using $(n+1)\delta$ space.  Either way, the space required will not exceed $max\set{5, n}\cdot\delta+S(k-1)=S(k)$.

\end{proof}

\begin{theorem}\label{thm:complexity}
The complexity of determining satisfiability of a temporal formula over two dimensional Minkowski spacetime (either irreflexive or reflexive) is  PSPACE-complete.  
\end{theorem}
\begin{proof}
By lemma~\ref{lem:4square}, a temporal formula $\phi$ is satisfiable in $(\reals^2, <)$ iff there are four boundary maps  $\partial_1, \partial_2, \partial_3, \partial_4\in B$ such that $(\partial_1\oplus_E\partial_2)\oplus_N(\partial_3\oplus_E\partial_4)$ is an open boundary map and $\phi\in\partial_1(\t) \; (=\partial_2(\l)=\partial_3(\r)=\partial_4(\b))$.  By lemma~\ref{lem:alg2} this can be checked non-deterministically by algorithm~\ref{alg2} using only space $4\times S(n)=O(n^5)$.  Since NPSPACE=PSPACE it follows that satisfiability (and validity) over $(\reals^2, <)$ is in PSPACE.   Validity over $(\reals^2, \leq)$ reduces to validity over $(\reals^2, <)$ so the former problem is also in PSPACE. On the other hand,
 the purely modal logic of ${(\reals^2, \leq)}$ (with no past operator) is {\bf S4.2} \cite{Gol80}, and this logic is known to be PSPACE hard \cite[theorem~20]{Shap05}, based on \cite{Lad77}.  Trivially, the purely modal logic of $(\reals^2, \leq)$ reduces to the temporal logic over $(\reals^2, \leq)$ which reduces to the temporal logic over $(\reals^2, <)$, so these temporal logics are PSPACE-hard.
\end{proof}

\section{Temporal Logics of Intervals}

These decidability results can be modified to work for certain temporal logics of intervals.
\begin{theorem}
Let $D$ be the open half plane $\set{(x, y)\in\reals^2: y>x}$ and let $\leq, <$ denote the restrictions to $D$  of the relations on $\reals^2$ defined in section~\ref{sec:prelims}.  Validity over $(D, <)$  and validity over $(D, \leq)$  are PSPACE-complete.

\end{theorem}
\begin{proofsk}
Its not hard to show that a purely modal formula is satisfiable in $(\reals^2, \leq)$ if and only if it is satisfiable in $(D, \leq)$, so we get PSPACE-hardness by the results cited above.

To show that the validity problem for the temporal logic over $(D, <)$ is in PSPACE, the proof is similar to the proofs of theorem~\ref{thm:decidable} and  lemma~\ref{lem:alg2}, but as well as rectangle models we also consider triangle models. A \emph{triangle} has vertices $(a,a), (b, b), (a, b)$ (some $a<b$) and it is either $\set{(x, y): y>x,\; a<x, y<b}$ (open), $\set{(x, y):y>x,\; a<x<b, \; a<y\leq b}$ (includes $N$), $\set{(x, y):y>x,\; a\leq x< b,\; a<y <b}$ (includes $W$), or $\set{(x, y):x<y,\; a\leq x<b,\; a<y\leq b}$ (includes $N, W, \l$).  A triangular map is a coherent function from a triangle to \MCSs  where all defects are passed to boundary points, as in definition~\ref{def:rect}. A \emph{triangular boundary map} (see figure~\ref{fig:triangle}) is a partial coherent map $\triang$ from $\set{N, W}\cup\set{\l}\cup\set{-, +}$ to traces (for the first two),  \MCSs (for the next one) and clusters (for the last two), whose domain is either $\set{-, +}, \set{-, +, N}, \set{-,+, W}$ or $\set{-, +, N, W, \l}$,  where every future defect of $\triang(+)$ is passed up to the final cluster of $\triang(N)$ and every past defect of $\triang(-)$ must be passed down to the initial cluster of $\partial(W)$,  cf. definition~\ref{def:bm}.  The height of a triangular map is the length of the longest chain of clusters from the lower to the upper cluster of the triangular map.  If $h$ is a triangular map then $\triang^h$ is the triangular boundary map induced by $h$.
\begin{figure}
\[
\xymatrix{
\l\ar@{-}[rr]^N&&{}\ar@{--}[lldd]\\
&&\\
{}\ar@{-}[uu]^W\ar@{}[rruu]^<(.8){+}\ar@{}[rruu]^<(.2){-}&&
} \hspace{1in}
\xymatrix{
\ar@{--}[rr]\ar@{}[rd]|\partial&\ar@{}[rd]|<(.25){\triang_2}&\\
\ar@{-}[r]\ar@{}[rd]|<(.25){\triang_1}&\ar@{-}[u]&\\
\ar@{-}[uu]\ar@{--}[rruu]&&
}
\]
\caption{\label{fig:triangle} (a) A triangular boundary map, (b) \label{fig:extra joins} the join of $\triang_1, \partial$ and $\triang_2$}
\end{figure}
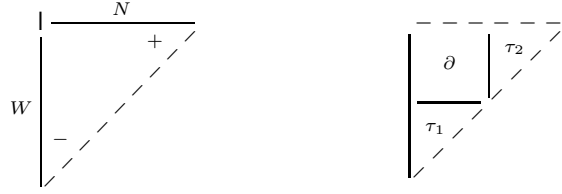
Then we define a set $T_0$ of triangular boundary maps $\triang$ satisfying $\triang(-)=\triang(+)$, where  all past defects of $\triang(N)$ are passed down to either $\triang(\l)$ or $\triang(+)$, and dual conditions.  If $\triang_1, \triang_2$ are triangular boundary maps, $\partial$ is a rectangular boundary map, $\partial(S), \partial(E)$ are defined,  $\triang_1(N)=\partial(S)$,\/  $\partial(E)=\triang_2(W)$, \/ $\triang_1(\l)=\partial(\b),\;\triang_2(\l)=\partial(\t)$   then we may define the triangular join of $\triang=\triang(\triang_1, \partial, \triang_2)$ to be the triangular boundary map illustrated in figure~\ref{fig:extra joins}(b), so $\triang(W)=\triang_1(W)+_{\triang_1(\l)}\partial(W)$, $\triang(+)=\triang_2(+)$, etc.

Suppose we have defined  a set $T_k$ of triangular maps of height at most $k$ and established that $\triang\in T_k$ iff there is a triangular map $h$ of height at most $k$ such that $\triang=\triang^h$ (easy to check for $k=0$).  Let $T_{k+1}$ be the union of $T_k$ and the set of triangular joins $\triang(\triang_1, \partial, \triang_0)$, where $\triang_1\in T_k$,\/ $\partial\in B$ and $\triang_0\in T_0$.       If $\triang\in T_{k+1}$ it is not hard to construct a triangular map $h$ such that $\triang=\triang^h$, by joining together triangular maps for $\triang_0, \triang_1$ and a rectangular map for $\partial$.  Conversely, let $h$ be a triangular map of height $k+1$.    Let $\x$ be the infimum of the set of points on the line $y=x$ such that $h\restr{E(\x)}$ has the same lower and upper cluster.  Then the height of $h\restr{E(\x)}$ is zero and the height of $h\restr{S(\x)}$ is no more than $k$ (see figure~\ref{fig:triangle}(b)).  Also, $h\restr{N(\x)\cap W(\x)}$ is a rectangle model.  Inductively $\triang_0=\triang^{h\restr{E(\x)}}\in T_0,\; \triang_1=\triang^{h\restr{S(\x)}}\in T_k$ and by lemma~\ref{lem:main}, $\partial = \partial(h, N(\x)\cap W(\x))\in B$,   Hence $\triang^h$ is the triangular join $\triang(\triang_1, \partial, \triang_0)$ and belongs to $T_{k+1}$.

A non-deterministic algorithm to check if $\triang$ belongs to $T_{k+1}$ guesses $\triang_0, \partial$ and $\triang_1$, checks that they join to make $\triang$,  checks $\triang_0\in T_0$, checks $\partial\in B$ and checks $\triang_1\in T_k$ and this can be done using only polynomial space.

To check if $\phi$ is satisfiable in $(D, <)$ we check if there is an open triangular boundary map in $T_n$ that may be decomposed as $\triang(\triang_1, \partial, \triang_0)$ and $\phi$ is satisfiable in a rectangle map whose boundary map is $\partial$ (use lemma~\ref{lem:4square}).  All this can  be done with polynomial space.
\end{proofsk}

Let $\mathcal I=\set{(x, y): x<y\in\reals}$ be the set of intervals with real endpoints and consider the following two binary relations over $\mathcal I$:
\begin{align*}
L&= \set{\mbox{before, meets, overlaps, ended\_by, starts}}\\
LE&= L\cup\set{\mbox{equals}}
\end{align*}
 i.e. for intervals $i, j$,\/ $(i, j)\in LE$ iff the start of $j$ is not before the start of $i$ and the end of $j$ is not before the end of $i$, and $L$ is the irreflexive restriction of $L$.
\begin{corollary}
The temporal logics over the frames $(\mathcal I, L)$ and $(\mathcal I, LE)$  are PSPACE-complete.
\end{corollary}
\begin{proof}
The frame $(D, <)$ is identical to the frame $(\mathcal I, L)$.  The frame $(D, \leq)$ is identical to the frame $(\mathcal I, LE)$.
\end{proof}

\begin{definition}\label{def:delta}
Let $\c P = (P, <)$ be a  finite  set with a transitive relation,  for each $p\in P$ consider the point as a propositional letter.  Let $\Delta(\c P)$ be the formula
\[
\begin{array}{l lll}
\G\H&[ & \bigvee P &\wedge \\
&&\bigwedge_{p\neq q\in P}\neg(p\wedge q)&\wedge\\
&&\bigwedge_{p\neq q,\;p<q\in P}(p\rightarrow\F q)\wedge(q\rightarrow\P p)&\wedge\\
&&\bigwedge_{p\neq q,\;p\not< q\in P}(p\rightarrow\G\neg q)\wedge(q\rightarrow \H\neg p)&\\
&]
\end{array}
\]
\end{definition}
Note that $\Delta(\c P)$ does not contain either $p\rightarrow \F p$ or $p\rightarrow \G\neg p$, for $p\in\c P$, making the formula indifferent to reflexivity.
We also define a formula $\Delta^+(\c P)$ in a similar way, but here if $p$ is an irreflexive point we include an additional conjunct $p\rightarrow \G\neg p$ and if $p$ is a reflexive point we include a conjunct $p\rightarrow (\F p\wedge \P p)$ in the scope of $\G\H$.  $\Delta^+(\c P)$ has a strong opinion about reflexivity.

Since these formulas are in the scope of $\G\H$,  which acts as a universal modality for directed frames, the point of evaluation of $\Delta(\c P)$ or $\Delta^+(\c P)$ makes no difference, so we may write $\c F\models \Delta(\c P)$ instead of $\c F, \x\models\Delta(\c P)$, where $\x$ is a point in the directed frame $\c F$. 
Clearly, $\Delta(\c P)$ and $\Delta^+(\c P)$ are true in the frame $\c P$ under the valuation which makes the proposition $p$ true at the node $p$ and nowhere else (for each $p\in \c P$).  Furthermore, if $\c P'$ is obtained from $\c P$ by adding or subtracting reflexive edges then $\c P'\models\Delta(\c P)$.

\medskip

We now consider another half-plane $\set{(x, y)\in\reals^2:x+y>0}$ and a bounded triangular frame $(A, <)$ where the base set is
\[A= \set{(x, y): 0<x, y<1,\; x+y>1}.\]
The unbounded real half plane $\set{(x, y)\in\reals^2:x+y>0}$ and this triangular frame are equivalent (essentially, by lemma~\ref{lem:equiv}).  We consider the validity problem for 
 $(A, <)$.  The trick we used for the half-plane $D$ won't work this time because the boundary line $x+y=1$ is not time-like and cannot in general be described by a finite ordered sequence of clusters and \MCSss.  Also note that the past directed confluence axiom 
\[ \P\H p\rightarrow \H\P p\]
can fail in $(A, <)$ whereas it is valid over  $(\reals^2, <)$ and $(D, <)$.  However, we provide a $p$-time reduction of the validities of $(A, <)$ to those of the closed unit square $\Sq= ([0,1]\times[0,1], <)$.    The validity problem over $\Sq$ can be checked with polynomial space, essentially by using algorithm~\ref{alg2}, as we argue below.

Let $\b, \t, \l, \r, N, S, E, W, \delta_1, \delta_2, \alpha, \beta$ be propositions that do not occur in $\phi$ (where the first four are intended to represent the four corners of the square, the next four are intended for the four sides of the square,  $\delta_1, \delta_2$ are intended for the diagonal line from $\l$ to $\r$, \/ $\alpha$ and $\beta$ are intended for the interior of the square above and below the diagonal, respectively).  Let
$\c P$ be the transitive relation illustrated on the left in figure~\ref{fig:square}.  Let $v_0:\props\rightarrow\wp([[0,1]\times[0,1])$ be the valuation over $\Sq$ illustrated in figure~\ref{fig:square} on the right, defined by 
\begin{align*}
v_0(\b) &=\set{(0,0)}& v_0(\t)&=\set{(1,1)}\\
v_0(\l)&=\set{(0,1)}&v_0(\r)&=\set{(1,0)}\\
v_0(S)&=\set{(x, 0):0<x<1}&v_0(W)&=\set{(0, y):0<y<1}\\
v_0(N)&=\set{(x, 1):0<x<1}&v_0(E)&=\set{(1, y):0<y<1}\\
v_0(\alpha)&=\set{(x, y):x+y>1,\; x, y<1}&v_0(\beta)&=\set{(x, y):x+y<1,\; x, y>0}\\
v_0(\delta_1)&= d_1&v_0(\delta_2)&=d_2
\end{align*}
where $d_1, d_2$ form a dense partition of the diagonal $\set{(x, y):x+y=1,\; x, y>0}$. 
\begin{lemma}\label{lem:delta}
 $(\Sq, v_0)\models\Delta^+(\c P)$.  Conversely, let $w$ be any valuation such that $(\Sq,  w)\models\Delta^+(\c P)$.  Let $w'$ be obtained from $w$ by reflection in $y=x$, i.e. $w'(p)=\set{(y, x): (x, y)\in w(p)}$, for each proposition $p$.   Then either $w$ or $w'$ agrees with $v_0$ on $\b, \t, \l, \r, N, S, E, W$, and $w(\delta_1)\cup w(\delta_2)$ is a continuous spatial open line segment  $d$ from $(0,1)$ to $(1,0)$,\/ $w(\alpha)$  is the interior of the square above $d$ and $w(\beta)$ is the interior of the square below $d$.
 \end{lemma}
 \begin{proof}
 The first part is easily verified.  For the second part, since $\Delta^+(\c P)\models \b\rightarrow\H\bot$ we must have $w(\b)=\set{(0, 0)}$ and similarly $\w(\t)=\set{(1, 1)}$.  Since $\Delta^+(\c P)$ entails $\b\rightarrow\G(\neg \F\l\vee\neg\F\r)$ and $\t\rightarrow\H(\neg\P\l\vee\neg\P\r)$ it follows that $w(\l)$ is either $\set{(0, 1)}$ or $\set{(1, 0)}$ and $w(\r)$ is the other singleton.  It is then easy to see that $(\Sq, w)\models\Delta^+(\c P)$ forces either  $w$  or $w'$ to agree with $v_0$ on $\b, \t, \l, \r, N, S, E, W$.  Since $w(\beta)$ is downward closed within the interior of the square, by lemma~\ref{lem:segment} its upper boundary is homeomorphic to a line segment (clearly it contains more than just one point), similarly $w(\alpha)$ is upward closed within the interior of the square and its lower boundary is homeomorphic to an open line segment.   All points in the open set between these boundary lines must be within $w(\delta_1)\cup w(\delta_2)$, but since $\delta_1, \delta_2$ are irreflexive in $\c P$ there can be no such points, hence these two boundaries are identical.  Since $\beta\rightarrow \F \delta_1$ and $\beta\rightarrow \F \delta_2$ are included in $\Delta^+(\c P)$ it follows that this common boundary line is spatial and densely covered by $w(\delta_1)$ and $w(\delta_2)$. Since every  point in $w(N)\cup w(E)$ is above points in $w(\delta_1)$ and  points in $w(\delta_2)$ and every point in $w(W)\cup w(S)$ is below points in $w(\delta_1)$ and points in $w(\delta_2)$, it follows that the common boundary line runs from $(0,1)$ to $(1, 0)$.
 \end{proof}

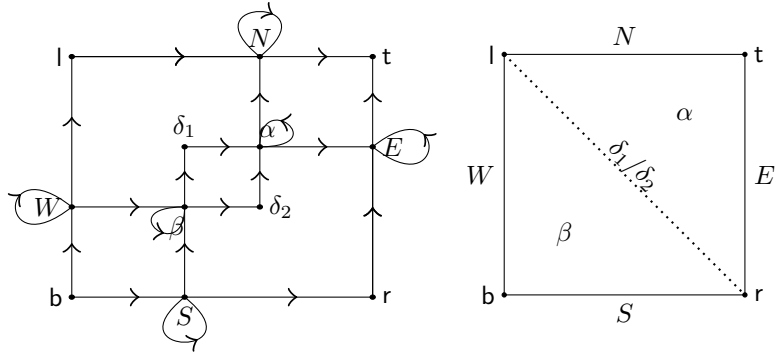
\begin{figure}
\begin{tikzpicture}[xscale=.5,yscale=.4]

\draw[fill] (0,0) circle [radius=2pt];
\draw[fill] (8,0) circle [radius=2pt];

\draw[fill] (0,8) circle [radius=2pt];

\draw[fill] (0,3) circle [radius=2pt];

\draw[fill] (3,0) circle [radius=2pt];
\draw[fill] (3,3) circle [radius=2pt];
\draw[fill] (5,8) circle [radius=2pt];
\draw[fill] (8,5) circle [radius=2pt];
\draw[fill] (5,5) circle [radius=2pt];
\draw[fill] (5,3) circle [radius=2pt];
\draw[fill] (3,5) circle [radius=2pt];
\draw[fill] (8,8) circle [radius=2pt];

\node[left] at (0,0) {$\b$};
\node[right] at (8,8) {$\t$};
\node[left] at (0, 8) {$\l$};
\node[right] at (8,0) {$\r$};
\node[below] at (3, 0) {$S$};
\node[left] at (0, 3) {$W$};
\node[above] at (5, 8) {$N$};
\node[right] at (8, 5) {$E$};
\node[below] at (2.8, 3){$\beta$};
\node[above] at (3, 5) {$\delta_1$};
\node[right] at (5, 3) {$\delta_2$};
\node[above] at (5.2, 5) {$\alpha$};

\draw[->-] (3,0)--(3,3);
\draw[->-] (0,3)--(3,3);
\draw[->-] (3,3)--(5,3);
\draw[->-] (3,3)--(3,5);
\draw[->-] (3,5)--(5,5);
\draw[->-] (5,3)--(5,5);
\draw[->-] (5,5)--(5,8);
\draw[->-] (5,5)--(8,5);

\draw[->-] (0,0)--(0,3);
\draw[->-] (0,0)--(3,0);

\draw[->-] (0,3)--(0,8);
\draw[->-] (3,0)--(8,0);
\draw[->-] (8,0)--(8,5);
\draw[->-] (0,8)--(5,8);

\draw[->-] (8,5)--(8,8);
\draw[->-] (5,8)--(8,8);

\draw[->-] (8,0)--(8,5);
\draw[->-] (0,8)--(5,8);

\draw[->-] (5,8)..controls (3,10) and (7,10) .. (5,8);
\draw[->-]  (8,5) .. controls (10,3) and (10,7) .. (8,5);
\draw[->-]  (3,0)..controls (1,-2) and (5,-2) .. (3,0);
\draw[->-]  (0,3) .. controls (-2,1) and (-2,5) .. (0,3);
\draw[->-] (5,5) .. controls(7,5) and (5,7) ..(5,5);
\draw[->-] (3,3) .. controls  (1,3) and (3,1) .. (3,3);

\end{tikzpicture}
\begin{tikzpicture}[scale=.4]
\draw  (0,0) rectangle (8,8);

\draw[fill] (0,0) circle [radius=2pt];
\draw[fill] (8,0) circle [radius=2pt];
\draw[fill] (0,8) circle [radius=2pt];
\draw[fill] (8,8) circle [radius=2pt];

\draw[dotted,thick] (0,8)--(8,0);

\node[left] at (0,0) {$\b$};
\node[right] at (8,8) {$\t$};
\node[left] at (0, 8) {$\l$};
\node[right] at (8,0) {$\r$};

\node[below] at (4, 0) {$S$};
\node[left] at (0,4) {$W$};
\node[right] at (8,4) {$E$};
\node[above] at  (4,8) {$N$};

\node at (2,2) {$\beta$};
\node at (6,6) {$\alpha$};
\node[rotate=-45] at(4.3,4.3) {$\delta_1/\delta_2$};

\node at (4,-1.8) {\/};

\end{tikzpicture}

\caption{\label{fig:square}  The transitive relation $\c P$ and a model $(\Sq, v_0)$ of $\Delta^+(\c P)$.  }

\end{figure}
Now, for the reduction of the validities of $(A, <)$ to those of $\Sq$, let $\phi$ be any temporal formula not involving any of the propositions in $\c P$.   Recall from definition~\ref{def:rel} that $\phi_\alpha$ is obtained from $\phi$ by  relativising to $\alpha$.   Our reduction $\rho$ maps $\phi$ to 
\[\rho(\phi)=\Delta^+(\c P)\wedge \phi_\alpha.\]
\begin{lemma}
Let $\phi$ be a temporal formula.  $\phi$ is satisfiable in $(A, <)$ if and only if $\rho(\phi)$ is satisfiable in $\Sq$.
\end{lemma}

\begin{proof}
Suppose $\phi$ is satisfiable in $(A, <)$, say $(A, <, v), \x\models\phi$ (some valuation $v$, some $\x\in A$).  Now let $v'$ be identical to $v$ on propositions occurring in $\phi$ but  identical to $v_0$ on propositions $\b, \t, \l, \r, N, S, E, W, \delta_1, \delta_2, \alpha, \beta$.  Then $(\Sq, v)\models \rho(\phi)$, by
lemmas~\ref{lem:delta} and \ref{lem:P}.

Conversely suppose $(\Sq, v), \x\models\rho(\phi)$, some $\x$ necessarily in $v(\alpha)$.   By lemma~\ref{lem:P} we may assume that $v(s)\subseteq v(\alpha)$, for each proposition $s$ in $\phi$.  By lemma~\ref{lem:delta} we know that $v(\delta_1)\cup v(\delta_2)$ is a spatial line from $(0,1)$ to $(1, 0)$ so by lemma~\ref{lem:equiv} there is an equivalent structure $(\Sq, v')\models\rho(\phi)$ where $v'(\delta_1)\cup v'(\delta_2)$ is the line $x+y=1$ so $v'(\alpha)=A$.  Recall  from definition~\ref{def:rel} that $(A, <, v'_A)$ is the structure obtained  from $(\Sq, v')$ by relativising to $A$.  By lemma~\ref{lem:P} we deduce that $(A, <, v'_A)\models\phi$, so $\phi$ is satisfiable in a model on the frame $(A, <)$.
\end{proof}
Hence,
\begin{theorem}
The interval logics $(\mathcal I, \subset)$ and $(\mathcal I, \subseteq)$ are PSPACE-complete, where $\subseteq$ is the containment relation between intervals $\set{\mbox{equals, starts, during, ends}}$ and $\subset$ is the strict containment relation $\set{\mbox{starts, during, ends}}$.  \end{theorem}

\begin{proof}
First we check that satisfiablity (hence also validity)  over the frame $\Sq$ is decidable in PSPACE.  To check if a temporal formula $\phi$ is satisfiable in $\Sq$ we check if there are four closed proper boundary maps $\partial_1, \partial_2, \partial_3, \partial_4\in B$ (using algorithm~\ref{alg2}) that fit together at a common corner containing $\phi$, i.e. $\phi\in \partial_1(\r)=\partial_2(\b)=\partial_3(\t)=\partial_4(\l)$, and where the join $(\partial_3\oplus_E\partial_4)\oplus_N(\partial_1\oplus_E\partial_2)$ has no external defects.  This can be checked using polynomial space, by lemma~\ref{lem:alg2}.

 Let $\lambda$ be an  order isomorphism from $(\reals, <)$ onto $((0, 1), <)$.
The map $f:\mathcal I \rightarrow A$ defined by $f(x, y)=(1-\lambda(x), \lambda(y))$ is an isomorphism from $(\mathcal I, \subset)$ onto $(A, <)$.  Since the validity problem for $(A, <)$ reduces to the validity problem for $\Sq$ which is in PSPACE, it follows that the interval logic of strict containment is in PSPACE.  The satisfiability problem for $(\mathcal I, \subseteq)$ reduces to the satisfiability problem for $(\mathcal I, \subset)$ so the interval logic of (non-strict) containment is also in PSPACE.  The purely modal logic $(\mathcal I, \subseteq)$ is {\bf S4.2}, so these temporal logics are PSPACE-hard, by \cite[theorem~20]{Shap05}.
\end{proof}

\section{Distinguishing formulas}\label{sec:distinguish}
The formula $p\wedge \G\neg p$ is satisfiable in any frame with an irreflexive point, but not in reflexive frames, the past operator $\P$ is not needed in order to distinguish these two kinds of frames.  We mentioned at the beginning that the purely modal language cannot distinguish between two or three dimensional reflexive  Minkowski spacetime, nor is there a difference between the modal logics of these frames and the modal logic of a similar frame where the coordinates are restricted to the rational numbers --- in each case the modal logic is {\bf S4.2}.  However, with temporal operators the logics of these different frames are not the same.      Recall from definition~\ref{def:delta} the formula $\Delta(\c P)$, for any transitive relation $\c P$.

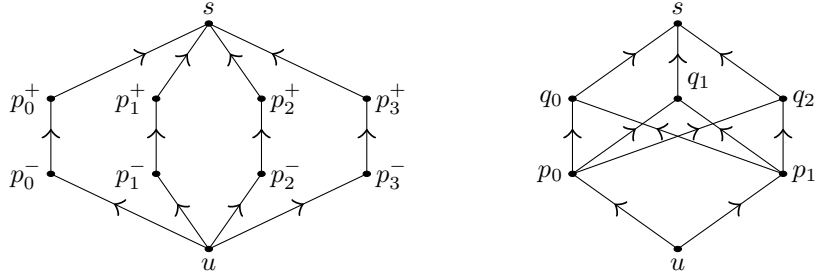
\begin{figure}
\begin{tikzpicture}[xscale=.7,yscale=.5]
\draw[fill] (0,0) circle [radius=2pt];
\draw[fill] (-3,2) circle [radius=2pt];
\draw[fill] (-1,2) circle [radius=2pt];
\draw[fill] (1,2) circle [radius=2pt];
\draw[fill] (3,2) circle [radius=2pt];

\draw[fill] (-3,4) circle [radius=2pt];
\draw[fill] (-1,4) circle [radius=2pt];
\draw[fill] (1,4) circle [radius=2pt];
\draw[fill] (3,4) circle [radius=2pt];

\draw[fill] (0,6) circle [radius=2pt];

\draw[->-] (0,0) -- (-3,2);
\draw[->-] (0,0) -- (-1,2);
\draw[->-] (0,0) -- (1,2);
\draw[->-] (0,0) -- (3,2);

\draw[->-] (-3,2) -- (-3,4);
\draw[->-] (-1,2) -- (-1,4);
\draw[->-] (1,2) -- (1,4);
\draw[->-] (3,2) -- (3,4);

\draw[->-] (-3,4) -- (0,6);
\draw[->-] (-1,4) -- (0,6);
\draw[->-] (1,4) -- (0,6);
\draw[->-] (3,4) -- (0,6);

\node[below] at (0,0) {$u$};
\node[left] at (-3,2) {$p_0^-$};
\node[left] at (-1,2) {$p_1^-$};
\node[right] at (1,2){$p_2^-$};
\node[left] at (-3,4) {$p_0^+$};
\node[left] at (-1,4) {$p_1^+$};
\node[right] at (1,4){$p_2^+$};

\node[right] at (3,2) {$p_3^-$};
\node[right] at (3, 4){$p_3^+$};
\node[above] at (0,6) {$s$};
\end{tikzpicture}
\hspace{.5in}
\begin{tikzpicture}[xscale=.7,yscale=.5]
\draw[fill] (0,2) circle [radius=2pt];
\draw[fill] (2,0) circle [radius=2pt];
\draw[fill] (4,2) circle [radius=2pt];
\draw[fill] (0,4) circle [radius=2pt];
\draw[fill] (4,4) circle [radius=2pt];
\draw[fill] (2,6) circle [radius=2pt];
\draw[fill] (2,4) circle [radius=2pt];

\draw[->-] (2,0) -- (0,2);
\draw[->-] (2,0) -- (4,2);
\draw[->-] (0,2) -- (0,4);
\draw[->-] (0,2) -- (4,4);
\draw[->-] (4,2) -- (4,4);
\draw[->-] (4,2) -- (0,4);
\draw[->-] (0,4) -- (2,6);
\draw[->-] (4,4) -- (2,6);
\draw[->-] (0,2) -- (2,4);
\draw[->-] (4,2) -- (2,4);
\draw[->-] (2,4) -- (2,6);

\node[below] at (2,0) {$u$};
\node[above] at (2,6) {$s$};
\node[left] at (0,2) {$p_0$};
\node[right] at (4,2){$p_1$};
\node[left] at (0,4) {$q_0$};
\node[above right] at (2,4) {$q_1$};
\node[right] at (4,4) {$q_2$};

\end{tikzpicture}

\caption{\label{fig:spos}  Partial orders: $\c P_1$ (left) is a non-distributive lattice and $\c P_2$ (right) is not even a lattice.  $\Delta(\c P_1)$ is satisfiable in two dimensional rational frames, but not in two dimensional real frames.  $\Delta(\c P_2)$ is satisfiable in three dimensions but not in two dimensions.}
\end{figure}

\begin{theorem}\label{thm:po}
Let $\c P_1, \c P_2$ be the  partial orders shown in figure~\ref{fig:spos}.
\begin{enumerate}
\item The formula $\Delta(\c P_1)$ is satisfiable in $(\rats^2, \leq)$ and $(\rats^2, <)$ but not in $(\reals^2, \leq)$ nor in $(\reals^2, <)$.
\item 
The formula $\Delta(\c P_2)$ is satisfiable in $(F^3, \leq)$ and $(F^3, <)$ but not in $(F^2, \leq)$ or $(F^2, <)$, where $F$ is either $\reals$ or $\rats$.
\end{enumerate}
\end{theorem}

\begin{proof}
To see that $\Delta(\c P_1)$ is satisfiable in $(\rats^2, \leq)$ and $(\rats^2, <)$  we define a valuation $v:\props\rightarrow\wp(((0,1)\times(0,1))\;\cap\;\rats^2)$ which satisfies $\Delta(\c P_1)$, regardless of whether the accessibility relation is reflexive or irreflexive.  Essentially, this model is built like a Cantor set.

 To start with, let $v(u)=\set{(x, y): x+y<1,\; x, y\in (0, 1)\cap\rats}$ and $v(s)=\set{(x, y):x+y>1,\; x, y\in (0, 1)\cap\rats}$.    The diagonal line $x+y=1,\; 0<x<1$ is considered as  a single open gap.  Points near this diagonal may get reassigned as the construction proceeds.
   At a finite stage of the construction there will be a finite number of open gaps along the diagonal.  At the next stage we have two choices.   We may pick some $i<4$ and place a closed square in the central third of the open gap, with the centre of the new square on the main diagonal.   Every point in the new square is now added to $v(p_i^-)$ except the top corner which is added to $v(p_i^+)$, removing all points in the new square and below the main  diagonal from $v(u)$ and all points in the new square above the diagonal from $v(s)$.  Alternatively, pick any point in any open gap (there are countably many such points)  and create a closed square  within the gap centred on the diagonal,  covering the chosen point, with all points of the square  added to $v(p_0^-)$ except the top corner which is added to $v(p_0^+)$ (removing off diagonal points from $v(u)/v(s)$ as before).  In either case a single open gap is replaced by two smaller open gaps, with the new closed square between them.  Continue in this way ensuring that within every open gap, for each $i<4$, eventually a closed square using $p_i^-, p_i^+$ is inserted and every point along the diagonal gets covered by a closed square eventually.  Since there are only countably many points and countably many open segments created in the process, this can be scheduled.  The limit of this process (well defined, since each point gets reassigned at most once) is the valuation $v$.  By this construction, every point in $v(u)$ is below a square whose bottom corner is in $v(p_i^-)$ (for each $i<4$) and every point in $v(p_i^-)$ is in a closed square whose top corner is in $v(p_i^+)$, which in turn is below points in $v(s)$, furthermore all the closed squares placed along the diagonal are incomparable.  Hence $\Delta(\c P_1)$ is true in  this rational model, whether the accessibility relation is reflexive or irreflexive.
    
 However, we show next that $\Delta(\c P_1)$ cannot be satisfied in $(\reals^2, \leq)$ or $(\reals^2, <)$.   
 Suppose instead that $\Delta(\c P_1)$ could be satisfied in $(\reals^2, \leq)$ or $(\reals^2, <)$, let $v$ be a valuation satisfying  $\Delta(\c P_1)$ (in either reflexive or irreflexive frame).  Pick any $\x_0\prec\x_1$ with $\x_0\in v(u)$ and $\x_1\in v(s)$.  Let $\Gamma$ be the upper boundary of $v(u)\cap [\x_0, \x_1]$, it is homeomorphic to a closed line segment, by lemma~\ref{lem:segment}.  Suppose $\Gamma\cap v(u)\neq\emptyset$, let
 $\x\in v(u)\cap\Gamma$.  There are incomparable points $\y_i\geq \x$ with $\y_i\in v(p_i^-)$.  Since $\x\in\Gamma$ we have $S=\set{\y:\x\prec\y}\subseteq\bigcup_{i<4}(v(p_i^-)\cup v(p_i^+))\cup v(s)$.
 For each $i<4$, either $\y_i>_1\x,\; \y_i>_2\x$ or $\y_i\succ\x$, so there are distinct $i, j$ such that $\y_i, \y_j$ both belong to $\set{\y:\y>_1\x}, \set{\y:\y>_2\x}$ or $S$.   Since $\y_i, \y_j$ are incomparable they cannot belong to the same ordered line segment, hence they both belong to $S$.  But if $\y_i\succ \x\in\Gamma$ and $\y_i\in v(p_i^-)$ then $\x\prec\z\prec\y_i\rightarrow \z\in v(p_i^-)$, hence $\y_j\in S$ must be above some $\z\in v(p_i^-)$, contradicting $\Delta(\c P_1)$.   We deduce that  $\Gamma\cap v(u)=\emptyset$, similarly, $v(s)\cap[\x_0, \x_1]$ is disjoint from its lower boundary.  Thus $\Gamma\subseteq\bigcup_{i<4}(v(p_i^-)\cup v(p_i^+))$.  
Let $M(\Gamma)=\set{\x\in\Gamma:\y<\x\rightarrow\y\not\in\Gamma}$, the set of minimal points of $\Gamma$.  For each $\x\in\Gamma$ there is $\y\leq\x$ with $\y\in M(\Gamma)$, since $\Gamma$ is closed, furthermore $\y$ is unique (if $\y_1, \y_2\in M(\Gamma)$, \/ $\y_1, \y_2\leq x$ and $\x\in v(p_i^-)\cup v(p_i^+)$, some $i<4$, then points between $\y_1$ and $\y_2$ cannot see witnesses of $p_j^-$ in their future, for $j\neq i<4$).  Write $m(\x)$ for the unique point in $M(\Gamma)$ below $\x\in\Gamma$.   Define an equivalence relation $\approx$ over $\Gamma$ by letting $\x\approx\y\iff m(\x)=m(\y)$.  By closure of $\Gamma$ the supremum of $\set{\y\in\Gamma:\y\geq_1\x}$ is in $\Gamma$ and above $m(\x)$, and similarly for $\set{\y\in\Gamma:\y\geq_2\x}$.  Hence, each $\approx$-equivalence class is homeomorphic to a closed line segment. By lemma~\ref{lem:point}, uncountably many equivalence classes are singleton points.  Let $\x=m(\x)$ be a non-extremal singleton point, say $\x\in v(p_i^-)$ (some $i<4$).  There is $\x^+>\x$ with $\x^+\in v(p_i^+)$.  But then,  there is $\y$ with either  $\x<_1\y\leq \x^+$ or $\x<_2\y\leq\x^+$, either way we get $\y\in v(p_i^-)\cup v(p_i^+)$. Suppose $\y\not\in\Gamma$, so it is in the interior of $v(p_i^-)\cup v(p_i^+)$.  Then  there are incomparable points $\z_1, \z_2<\y$ with $\z_1, \z_2\in v(p_i^-)\cup v(p_i^+)$ and $\z_1\wedge\z_2\in v(u)$, so $\z_1\wedge\z_2$ cannot see $p_j^-$ in its future, for $j\neq i<4$, a contradiction.  We conclude that $\y\in\Gamma$.  Since $\y$ is distinct from $\x$ this contradicts the fact that the equivalence class of $\x$ is a singleton.

 \medskip
 
 For the second part we construct a model  for $\Delta(\c P_2)$ over $(F^3, \leq)$  and over $(F^3, <)$ (where $F$ is either $\reals$ or $\rats$)  by defining a satisfying propositional valuation $v$.   We use cylindrical coordinates $[r, \theta, t]$ where $r, \theta$ are the modulus and argument of the spatial part and $t$ is the time coordinate.    The surface of the future lightcone from the origin $O$ is given by $r=t,\; t> 0,\; 0\leq\theta<2\pi$. Let $\lambda:[0, 2\pi)\rightarrow\set{0, 1, 2}$ be a dense function (i.e. $0\leq a<b<2\pi,\; i<3 \rightarrow\exists c\; (a\leq c\leq b\wedge \lambda(c)=i\wedge tan^{-1}(c)\in\rats)$,  the requirement $tan^{-1}(c)\in\rats$ ensures that the line $r=t>0,\;\theta=c$ meets $\rats^3$ densely).  We let $v(s)$ be the interior of the future lightcone with apex at the origin $O$,  we let $v(q_0), v(q_1), v(q_2)$ partition the surface of the future lightcone not including $O$, let $v(p_0)=\set{O},\; v(p_1)$ be the set of points incomparable with $O$ and let $v(u)$ be the set of points strictly below $O$, more precisely:
\begin{align*}
[r, \theta, t]\in v(s)&\iff 0\leq r<t &&\mbox{(inside future lightcone)}\\
[r, \theta, t]\in v(q_{\lambda(\theta)})&\iff r=t>0&&\mbox{(on surface of future lightcone)}\\
[r, \theta, t]\in r(p_0)&\iff r=t=0&&\mbox{(origin)}\\
[r, \theta, t]\in v(p_1)&\iff r>|t|&&\mbox{(incomparable with origin)}\\
[r, \theta, t]\in v(u)&\iff  r\leq -t,\; t<0&&\mbox{(below orgin)}.
\end{align*}
In the case where $F=\rats$ we must restrict to points with rational coordinates, but no other modifications are needed.  
  By this construction, every point in $v(s)$ is above points in $v(q_i)$ for $i<3$ (and conversely every point in $v(q_i)$ is below a point in $v(s)$), every point in $v(q_i) $ is above the origin, so above a point in $v(p_0)$ and above points in $v(p_1)$ (and conversely), and every point is above a point in $v(u)$, furthermore there are no points in $v(q_i)$ above or below points in $v(q_j)$ for $i\neq j<3$  and there are no points in $v(p_0)=\set{O}$ above or below points in $v(p_1)$.   Every point in $F^3$ belongs to $v(a)$ for exactly one proposition $a$.  Hence $(F^3, \leq), v\models\Delta(\c P_2)$ and $F^3, <, v\models\Delta(\c P_2)$.
 
 Finally we must show that $\Delta(\c P_2)$ is not satisfiable in $(F^2, \leq)$ or  $(F^2, <)$.   Suppose, for contradiction, that $v:\props\rightarrow\wp(D)$ is a propositional valuation, where $D$ is either $\reals^2$ or $\rats^2$, such that $(D, \leq), v\models \Delta(\c P_2)$ or $(D, <), v\models\Delta(\c P_2)$.    We will fix $D, v$ and either $\leq$ or $<$ for the rest of this proof so we may write $\x\models\phi$ instead of $(D, \leq), v, \x\models\phi$ or $(D, <), v, \x\models\phi$, where $\x\in D$.   There must be $\x\in D$ such that $\x\models p_0$ and there must be incomparable points $\y_i>\x$ where $\y_i\models q_i$ for $i<3$.  For any three incomparable points in $D$, one of them is between the other two.    Without loss, $\y_1$ is between $\y_0$ and $\y_2$, so $\x\prec\y_1$.  Observe that for all $\z\leq \y_1$ either  $\z\leq\y_0,\;\z\leq \y_2$ or $\z\geq\y_0\wedge\y_2$, hence either 
  $\z\leq \y_0\wedge\y_1,\; \z\leq \y_1\wedge\y_2$ or $\x\leq \z\leq \y_1$.   For all $\z\in [\x, \y_1]$ we have $\z\models (p_0\vee q_1)$ and $\y_0\wedge\y_1\models p_0,\; \y_1\wedge\y_2\models p_0$.  It follows that  for all $\z\leq\y_1$ either $\z\models q_1,\; \z\models p_0$ or $\z\models u$.  In particular, there is no $\z\leq\y_1$ where $\z\models p_1$, contradicting the assumption that $\Delta(\c P_2)$ is true in this model.
\end{proof}
Observe that $\Delta(\c P_1),\; \Delta(\c P_2)$ are consistent with \Ax\ (since they are true in the irreflexive, transitive, directed frames shown in figure~\ref{fig:spos}) but they are not satisfiable in $(\reals^2, <)$ and $\Delta(\c P_2)$ is not satisfiable in $(\rats^2, <)$, by theorem~\ref{thm:po}.  It follows that \Ax\ is not complete for the validities of $(\reals^2, <)$ or $(\rats^2, <)$.

\medskip

We conclude with some open problems:
\begin{enumerate}
\item Find a sound and complete axiomatisation of the temporal validities of $(\reals^2, \leq)$ and of $(\rats^2, \leq)$, also the corresponding irreflexive frames.
\item Prove the decidability of satisfiability of temporal formulas over $(\reals^2, \prec)$ (slower than light). Prove the decidability of the temporal logic of intervals with accessibility relation `during'.
\item  Consider the logics of other two dimensional frames, e.g. $(\ints^2, \leq),\; {(\ints\times\rats, \leq)}$ etc., check the decidability of these logics and look for complete axiomatisations.
\item  We have seen that the logics of $(\reals^2, \leq)$ and $(\reals^3, \leq)$ are not the same.  Are there formulas satisfiable in $(\reals^m, \leq)$ but not in $(\reals^n, \leq)$ for distinct $m, n>2$?
\item Is the logic of $(\reals^3, \leq)$ decidable? (Conjecture:no.)
\end{enumerate}

\newcommand{\etalchar}[1]{$^{#1}$}

\end{document}